\documentclass[preprint,12pt]{elsarticle}
\usepackage{amsmath}
\usepackage[titletoc,title]{appendix}
\usepackage{amsbsy}
\usepackage{amssymb}
\usepackage{amsthm}
\usepackage{natbib}
\usepackage{anyfontsize}
\usepackage{epsfig}
\usepackage{epstopdf}
\usepackage{wrapfig}
\usepackage{color,cancel}
\usepackage{ulem}
\usepackage{fullpage}
\usepackage{float}
\usepackage{breqn}
\usepackage{graphicx}
\usepackage{caption}
\usepackage{subcaption}
\usepackage{bbm}
\usepackage{tikz}
\usepackage{enumitem}
\usepackage[dvipsnames]{xcolor}
\usepackage{url} 
\pdfoutput=1

\usepackage{hyperref}

\usepackage{amsmath,amsfonts,amssymb,color,graphicx}

\newtheorem{theorem}{Theorem}[section]
\newtheorem{algorithm}{Algorithm}[section]
\newtheorem{corollary}{Corollary}[section]
\newtheorem{lemma}{Lemma}[section]
\newtheorem{proposition}{Proposition}[section]

\newtheorem{remark}{Remark}[section]

\numberwithin{equation}{section}

\newcommand{\M}{\mathcal{M}}
\newcommand{\Hdiv}{\bfH_{\mathrm{div}}}

\newcommand{\bu}{{\bf u}}
\newcommand{\bH}{{\bf H}}
\newcommand{\bD}{{\bf D}}

\newcommand{\bP}{{\bf P}}

\newcommand{\bv}{{\bf v}}
\newcommand{\bw}{{\bf w}}

\newcommand{\be}{{\bf e}}

\newcommand{\bZ}{{\bf Z}}
\newcommand{\bff}{{\bf f}}
\newcommand{\bfg}{{\bf g}}

\newcommand{\bphi}{{\boldsymbol \phi}}

\newcommand{\bfH}{{\bf H}}

\newcommand{\Reynolds}{\mathrm{Re}}
\newcommand{\Divergence}{\nabla \cdot}

\newcommand{\bK}{{\bf K}}
\newcommand{\bX}{{\bf X}}

\newcommand{\bL}{{\bf L}}

\newcommand{\bB}{{\bf B}}

\newcommand{\bg}{{\bf g}}
\newcommand{\bJ}{{\bf J}}

\usepackage{color}
\date{}
\usepackage[section]{placeins}
\makeatletter
\AtBeginDocument{%
	\expandafter\renewcommand\expandafter\subsection\expandafter{%
		\expandafter\@fb@secFB\subsection
	}%
}

\journal{Mathematics and Computers in Simulation}

\begin{document}

\begin{frontmatter}
\title{Improved Arrow-Hurwicz method for Stationary Inductionless 
Magnetohydrodynamics System}

\author[METU]{Duygu Uludağ}
\ead{duygimath@gmail.com}

\author[Atilim]{Fatma G. Eroglu}
\ead{fatma.guler@atilim.edu.tr}

\author[Sharjah]{Aziz Takhirov\corref{mycorrespondingauthor}\fnref{footnote1}}
\cortext[mycorrespondingauthor]{Corresponding author}
\ead{atakhirov@sharjah.ac.ae}

\author[METU]{Song\"{u}l Kaya}
\ead{smerdan@metu.edu.tr}

\address[METU]{Department of Mathematics,  Middle East Technical University, 06800 Ankara, Türkiye}

\address[Atilim]{Department of Mathematics, At{\i}l{\i}m University, 06830, Ankara, Türkiye}

\address[Sharjah]{Department of Mathematics, University of Sharjah, UAE}

\begin{abstract}
In this work, we propose a new Arrow-Hurwicz iterative scheme designed to solve the steady inductionless magnetohydrodynamics system. The main feature of the proposed scheme is the introduction of new penalty terms in the current density equation. These terms play a central role in effectively controlling the unfavorable mixed terms that commonly arise in such formulations. The proposed method is shown to achieve geometric convergence. Numerical tests affirm the efficiency of the new scheme without compromising accuracy.
\end{abstract}



\begin{keyword}
Inductionless MHD equation \sep Arrow-Hurwicz method \sep iterative method

\MSC[2010] 65M60 \sep 76D05
\end{keyword}
\end{frontmatter}

\section{Introduction}
Inductionless magnetohydrodynamics (IMHD) is applied to model MHD scenarios in which the magnetic field generated by currents within the conducting fluid is negligibly small relative to the external magnetic field $\mathbf{B}$. In a nondimensional form, the corresponding system in a simply-connected domain $\Omega$ is:
 \begin{eqnarray}
 -\Reynolds^{-1} \Delta\bu + (\bu\cdot\nabla)\bu +\nabla p-\kappa\bJ\times\bB=\bff \mbox{ in} \, \Omega, \label{eq:m1}\\
          \bJ+\nabla\phi-\bu\times\bB={\bf{g}} \mbox{ in} \, \Omega, \label{eq:m2}
   \\
\nabla \cdot \bu=0,\quad \nabla \cdot \bJ=0 \mbox{ in} \, \Omega, \label{eq:m3}
\\
 \bu=\mathbf{0},  \quad \bJ \cdot{\bf {n}}=0 \mbox{ on} \, \Gamma, \label{eq:m4}
\end{eqnarray}
where $\bu$ is the fluid velocity, $p$ is the kinematic pressure, $\bJ$ is the current density, and $\phi$ is the electric potential. In addition, $\kappa$ and $\Reynolds$ denote the coupling number and the Reynolds number, respectively. $\bff$ and $\bfg$ are the forcing terms. Equation \eqref{eq:m2} is what remains after combining Ohm's Law with Faraday's law in a simply-connected domain, whereas the second equation of \eqref{eq:m3} is the charge conservation equation. For simplicity, we assume insulating boundary conditions for the electric current $\bJ$ and homogeneous Dirichlet conditions for the velocity field. The analysis developed herein could be easily applied to perfectly conducting boundary conditions as well. 
 
Many industrial processes are modelled by this system, such as MHD pumps, MHD generators \cite{MHDbook1,MHDbook2}, and test blanket modules in nuclear fusion reactors \cite{mistrangelo2008magnetohydrodynamic,mistrangelo2009influence}.

Given their importance, the IMHD equations \eqref{eq:m1}–\eqref{eq:m4} have been well known since the 1960s, and there is 
 a significant amount of literature dedicated to their numerical approximations. Among these studies, \cite{petersonmhd} examined the well-posedness of weak solutions alongside their finite element approximations. In their formulation, the current density field $\bJ$ is removed from the system, and the electric scalar potential $\phi$ is governed by a Poisson equation. Layton \cite{LaytonMHD} et al. developed a two-level method and proved $L^2$-error estimates. 
In \cite{PLANAS20112977}, Badia et al. proposed a stabilized finite element method for the unsteady IMHD model. They consider a monolithic approximation with stabilizing terms, which leads to a coupled linear system. The coupling is then dealt with effective preconditioning and iterative solvers that could handle high
Hartmann number problems. {Zhang and Ding ~\cite{zhang} studied three iterative methods for approximating the stationary IMHD equations. These schemes rely on a fully coupled approach, necessitating the solution of a saddle-point linear system at every iterative step. A related work on fully divergence-free schemes is studied in \cite{Zhang2022}.}

Here we develop a new Arrow–Hurwicz (AH) \cite{AH1958} method that approximates \eqref{eq:m1}-\eqref{eq:m4}. First applied for steady incompressible Navier--Stokes equations \cite{Temam79}, the AH method decouples velocity and pressure by relaxing the incompressibility constraint, akin to the artificial compressibility regularization used for unsteady flows \cite{guermond2015high}.  
Another key feature of the scheme is the inclusion of the difference between the Laplacians of the velocity at the current and previous iterations in the momentum equation, to accelerate convergence to a steady state. The judicious choice of the artificial viscosity parameter \( \tfrac{1}{\rho} \) is crucial, and analysis shows that \( \rho \) must be sufficiently small. The finite element approximation of the AH scheme was later performed for the incompressible Navier-Stokes system ~\cite{chen17,TC23}, MHD flows ~\cite{yangmhd}, natural convection problems \cite{Calcolo2026}, and for stationary, thermally coupled, incompressible MHD flows  \cite{kerammhdah}. The finite element AH technique for solving the stationary IMHD system was recently studied by Xia and Yang \cite{xia_arrow}. 

This paper presents a new approach for discretizing the solution of stationary IMHD equations. Firstly, the velocity–pressure system is treated using the AH scheme of \cite{TC23}, which has been shown to improve the algorithm of \cite{chen17} significantly. Secondly, we propose adding new terms in the current density equation inspired by the classical streamline diffusion method \cite{SD_Hughes_1982}. Specifically, we add a difference of terms of the form
\[
\frac{1}{\rho} \, \bB \times \left( \bJ \times \bB \right)
\]
at the current and the previous iterations. As shown in the analysis, these additional terms generate nonnegative energy and dissipative terms, allowing control over the unfavorable mixed terms that arise when certain variables are lagged. The numerical tests confirm the computational efficiency of our approach. Another important feature of our scheme is that it segregates the solution of all four field variables. These ingredients make our AH scheme very different from that of \cite{xia_arrow}. 

This manuscript is organized as follows. In the next Section \ref{sec:Notations}, we introduce relevant notations and recall some relevant results. Section \ref{sec:AH_IMHD} describes the novel AH iterative finite element method for solving steady IMHD equations and its convergence analysis. Section \ref{sec:Numerics} presents a series of numerical experiments designed to corroborate the theoretical analysis and highlight the performance of the proposed schemes. We end with some concluding remarks.

\section{Notation and Preliminaries}
\label{sec:Notations}
In this paper, we adopt the standard notation for Sobolev spaces and their associated norms \cite{Ada75}. $L^q(\Omega)$ will denote the space of $q$ integrable functions with the norm expressed as $\| \cdot \|_{L^q{(\Omega)}}.$ The inner product in the space $L^2(\Omega)$ will be denoted by $(\cdot,\cdot)$ and its norm by $\| \cdot\|$. The standard Sobolev space $H^k(\Omega)$, where $k$ is a positive integer, is equipped with the norm $\|\cdot\|_k$. The vectorial counterparts of these spaces are denoted in bold font. 

The relevant function spaces in this work are
\begin{align}
    \bX=\bH_0^1(\Omega) , Q=L_0^2(\Omega) ,\nonumber\\
    \bD= \bH_0({\rm div} ,\Omega), S=L_0^2(\Omega), \nonumber
\end{align} 
where $\bD$ is equipped with the standard norm of $\Hdiv$ space: 
$$  \|\bJ\|_{\Hdiv}:=(\|\bJ\|^2+\|\nabla \cdot\bJ\|^2)^{1/2}. $$
{The dual space of $\bH_0^1(\Omega)$ is denoted by $\bH^{-1}(\Omega)$ with norm
\begin{eqnarray}
    \|\bff\|_{-1}= \sup\limits_{\bv \in \bX}  \dfrac{(\bff,\bv)}{\|\nabla \bv\|} \nonumber
\end{eqnarray}}
The energy norms $(\bX,\bD)$ and $(Q,S)$ are defined by 
 \begin{equation}
      \|(\bu,\bJ)\|_1=(\|\nabla\bu\|^2+\|\bJ\|_{\Hdiv}^2)^{1/2}  , \hspace{0.5cm}
     \|(p,\phi)\|=(\|p\|^2+\|\phi\|^2)^{1/2}. \nonumber
\end{equation}
For a given magnetic field $\bB \in \bL^\infty(\Omega)$, we also introduce the following quasi-inner product and its corresponding semi-norm
\begin{eqnarray}
    (\bJ,\bK)_{\bB}:=\left( \bJ \times \bB,\bK \times \bB \right), \, \, \|\bJ\|_{\bB}:=\| \bJ \times \bB\|, \, \, 
. \nonumber
\end{eqnarray}
Then the weak formulation of \eqref{eq:m1}-\eqref{eq:m4} takes the following form: $\forall (\bv, q, \bK,\psi) \in (\bX,Q,\bD ,S) $,  find $ (\bu, p,\bJ,\phi) \in (\bX,Q,\bD,S)$ solving
\begin{align}
    \Reynolds^{-1} (\nabla \bu, \nabla \bv)  +  b(\bu , \bu,\bv) - \kappa(\bJ \times\bB,\bv)-(p,\nabla \cdot \bv)&=(\bff,\bv),\label{eq:ef1}
    \\
    (\nabla\cdot \bu,q)&=0, \label{eq:ef3}\\
    (\bJ,\bK)-(\phi, \nabla \cdot \bK)-(\bu\times\bB,\bK) &= (\bg ,\bK),\label{eq:ef2}\\
    (\nabla\cdot \bJ,\psi)&=0. \label{eq:ef4}
\end{align}
Here the skew-symmetrized nonlinear term is (see \cite{v16}) :
\begin{equation*}
   b(\bu,\bv,\bw)=((\bu \cdot \nabla)\bv,\bw) + \dfrac{1}{2}\left((\nabla \cdot \bu)\bv,\bw\right). 
\end{equation*}

In the analysis of the paper, the following embedding inequalities will often be used
\cite{GR86,fortin1991mixed}: there exist positive constants $C_p$ and 
$C_4$ such that for all $\bv\in\bX$
\begin{eqnarray}
    \|\bv\| \leq C_p \| \nabla \bv\|, &  
    \|\bv\|_{L^4(\Omega)}\leq C_4\| \nabla \bv\|. \label{eq:po}
\end{eqnarray}

 Next, we recall some properties of the trilinear form: 
\begin{lemma} \label{lem:trilinear}
The following are true for $b(\cdot,\cdot, \cdot)$:
\begin{eqnarray*}
&&b(\bw,\bu, \bu)= 0,  \label{eq:nonlinear0*}\\
&&b(\bu,\bv, \bw) \leq \M \|\nabla \bu\|\|\nabla \bv\|\|\nabla \bw\|,\\
&&b(\bu,\bv,\bw)  \leq \M \|\nabla \bu \| \|\nabla \bv\| \|\nabla \bw\| + \dfrac{C_4}{2}\|\nabla \cdot \bu\| \|\bv\|_{L^4(\Omega)}\|\nabla \bw\|,\label{eq:nonlinearleq}
\end{eqnarray*}
for all $\bu, \bv,\bw \in \bX $, where $\M=\M(\Omega)$ is a positive constant.
 \end{lemma}
We also define the norm for source terms as
\begin{eqnarray*}
 \| \bf F\|_{*}={{(\|\bff\|_{-1}^2+\|\bg\|^2)}}^{1/2}.
\end{eqnarray*}
The following well-posedness holds for the problem \eqref{eq:ef1}-\eqref{eq:ef4}, cf. \cite[Theorem 3]{zhang}:
\begin{theorem}
Let $C_{\min}=\min \{\Reynolds^{-1},\kappa\}$. For $\bf {f} \in H^{-1}(\Omega)$ and $\bf{g} \in L^2(\Omega)$, if 
$ \sigma:=\dfrac{\M\|\bf F\|_{*}}{C_{\min}^2}< 1$ holds, then
the problem \eqref{eq:ef1}-\eqref{eq:ef4} is well-posed and the solution satisfies $\|(\bu,\bJ)\|_1\le \dfrac{\|\bf F\|_{*}}{C_{\min}}$.
\end{theorem}
\subsection{Finite Element Approximation}
\label{subsec:FE_IMHD}
Next, we examine the finite element approximation of the {stationary IMHD equations} \eqref{eq:ef1}-\eqref{eq:ef4}. Let $\mathcal{T}_h$ be a quasi-uniform and shape-regular triangular mesh for $d = 2$ or a tetrahedral mesh for $d = 3$ if $\Omega$. The local and the global mesh sizes are defined as $h_K = \mathrm{diam}(K)$ and  $h:= \max\limits_{K\in \mathcal{T}_h} h_K$, respectively. For any integer $k \ge 0, P_k(K)$ denotes the space of polynomials of degree $k$ on $K\in \mathcal{T}_h$, and we also set $\bP_k = P_k(K)^d$.
The conforming finite element spaces are $\bX_h \subset \bX, Q_h \subset Q, \bD_h \subset \bD$ and $S_h \subset S$. The well-posedness requires that these spaces satisfy the following assumptions; see, e.g., \cite{zhang}:
\begin{itemize}
\item Conforming finite-dimensional subspaces described above  satisfy the discrete inf-sup condition, i.e., there are  constants $\beta_s > 0$ and $ \beta_m > 0$, independent of $h$, such that
\begin{align}
  \inf\limits_{0\neq q_h \in Q_h}  \sup\limits_{0 \neq \bv_h \in  \bX_h} \dfrac{(q_h,\nabla \cdot \bv_h)}{\|\nabla \bv_h\| \|q_h\|} \ge \beta_s, \, 
  \inf\limits_{0\neq \psi_h \in S_h}  \sup\limits_{0 \neq \bK_h \in  \bD_h} \dfrac{(\psi_h, \nabla \cdot \bK_h)} {\| \bK_h\|_{\Hdiv} \|\psi_h\|} \ge \beta_m. \label{eq:inf}
\end{align}
The finite element spaces for velocity and pressure can be chosen as the Taylor-Hood finite element space $(P_2,P_{1})$, while for current density and electric potential, the $\Hdiv(\Omega)$ conforming pair $ (\bD_h,S_h) = (RT_k,P_k^{dc})$ is selected. That is, the $k$-th order Raviart-Thomas finite element space is employed for current density, and $k$-th order discontinuous piecewise linear elements are utilized for the electric potential, where $k=0$ or $k=1$.

Assuming enough regularity of the exact solutions at hand, these discrete spaces are known to satisfy the following approximation properties, cf. \cite{mixedfinitelement,GR86,petermonk}:
\begin{equation}    
\begin{aligned}
     \inf\limits_{\bv_h\in \bX_h}\|\bu-\bv_h\|_{1}&\leq Ch^{\gamma}\|\bu\|_{1+\gamma}, \\
     \inf\limits_{q_h\in Q_h}\|p-q_h\|&\leq Ch^{\gamma}\|p\|_{\gamma}, \\
     \inf\limits_{\bK_h\in \bD_h}\|\bJ-\bK_h\|_{\Hdiv}&\leq Ch^{k}(\|\bJ\|_{k}+\|\nabla \cdot \bJ\|_{k}), \\
     \inf\limits_{\psi_h \in S_h}\|\phi-\psi_h\|&\leq Ch^{k}\|\phi\|_{k},
\end{aligned}
\label{eq:ApproxmiationProperties}
\end{equation}
for $\bu \in \bX \cap \bH^{1+\gamma}(\Omega)$, $p \in Q \cap H^{\gamma}(\Omega)$, $\bJ \in \bD \cap \bH^{k}(\Omega)$ with $\Divergence \bJ \in \bH^{k}(\Omega)$, and $\phi \in S \cap H^{k}(\Omega)$.
\end{itemize}
Then the finite element approximation of the IMHD system is straightforward:

Find $(\bu_h,p_h,\bJ_h,\phi_h) \in (\bX_h,Q_h,\bD_h,S_h)$ such that $\forall \, (\bv_h,q_h,\bK_h,\psi_h) \in (\bX_h,Q_h,\bD_h,S_h)$ there holds:
\begin{eqnarray}
\Reynolds^{-1} (\nabla \bu_h, \nabla \bv_h)  + b(\bu_h,\bu_h,\bv_h) - \kappa(\bJ_h \times\bB,\bv_h) - (p_h,\nabla \cdot \bv_h) &=&(\bff,\bv_h)\label{eq:w1},\\
(\nabla \cdot \bu_h,q_h) & = &0, \\
(\bJ_h,\bK_h)-(\phi_h,\nabla \cdot \bK_h)-(\bu_h\times\bB,\bK_h) &=& (\bg ,\bK_h)\label{eq:w3},\\ 
(\nabla \cdot \bJ_h,\psi_h) &=&0.\label{eq:w4} 
\end{eqnarray}

Due to the choice of the finite element spaces, the following charge conservation property holds:
\begin{proposition}
   Let \((\bu_h, p_h, \bJ_h, \phi_h)\) be the solution to the system of equations \eqref{eq:w1}-\eqref{eq:w4}. Then, the scheme is charge-conserving, that is \(\nabla \cdot \bJ_h = 0\).
\end{proposition}
The following well-posedness of the problem \eqref{eq:w1}-\eqref{eq:w4} is can be shown \cite{zhang,Zhang2022}:
\begin{theorem}
If the uniqueness condition 
     \begin{equation}
         \sigma=\dfrac{\M\|\bf F\|_*}{C_{\min}^2}< 1, \label{c1}
     \end{equation}
 holds where $\M = \sup \dfrac{(\bu \cdot \nabla \bv, \bw)}{\| \nabla \bu \| \| \nabla \bv \| \| \nabla \bw\|}$, then the problem \eqref{eq:w1}-\eqref{eq:w4}  is well posed and 
     \begin{equation}
             \|(\bu_h,\bJ_h)\|_1 \leq  \dfrac{\|\bf F\|_*}{C_{\min}}. \label{eq:uJfembound}
     \end{equation}
     is satisfied.
 \end{theorem}
\begin{theorem} \label{exactbd}Let \eqref{c1} be satisfied and let $(\bu, p, \bJ, \phi)$ and $(\bu_h, p_h, \bJ_h, \phi_h)$ be the solutions of continuous
problem \eqref{eq:ef1}-\eqref{eq:ef4} and the discrete problem \eqref{eq:w1}-\eqref{eq:w4}, respectively. If $(\bu, p, \bJ, \phi)$ satisfies
\begin{equation*}
\bu \in \bX \cap \bH^{1+\gamma}(\Omega), \, p \in Q \cap H^{\gamma}(\Omega), \, \bJ \in \bD \cap \bH^{k}(\Omega) \text{ with } \Divergence \bJ \in \bH^{k}(\Omega), \, \text{ and } \phi \in S \cap H^{k}(\Omega),
\end{equation*}
then there holds 
\begin{align*}
    \|(\bu-\bu_h,\bJ-\bJ_h)\|_1+\|p-p_h\|&\leq C h^{\min\{\gamma,k\}}(\|\bu\|_{1+\gamma}+\|\bJ\|_{k}+\|\Divergence \bJ\|_{k} + \|p\|_{\gamma}),\nonumber\\
    \|\phi-\phi_h\| &\leq C h^{\min\{\gamma,k\}}(\|\bu\|_{1+\gamma}+\|\bJ\|_{k}+\|\Divergence \bJ\|_{k}+\|p\|_{\gamma}+\|\phi\|_{k}).
\end{align*}
\end{theorem}
\section{Improved AH scheme}
\label{sec:AH_IMHD}
This section presents the formulation of the novel method for approximating the IMHD system. 
\begin{algorithm}[\it Improved AH method] \label{algo1}
Let $(\bu_h^0,p_h^0,\bJ_h^0,\phi_h^0) = (0,0,0,0).$
For $n\geq 0$ and $\forall (\bv_h,q_h,\bK_h,\psi_h)\in (\bX_h,Q_h,\bD_h,S_h)$, find $(\bu_h^{n+1},p_h^{n+1},\bJ_h^{n+1},\phi_h^{n+1})\in (\bX_h,Q_h, \bD_h, S_h)$ so that 
\begin{eqnarray}
\lefteqn{\frac{1}{\rho_1} (\nabla(\bu_h^{n+1}-\bu_h^n),\nabla \bv_h) + \Reynolds^{-1} (\nabla \bu_h^{n+1}, \nabla \bv_h)  + b(\bu_h^{n},\bu_h^{n+1},\bv_h)
 \nonumber }\\
&& - \kappa(\bJ_h ^{n}\times\bB,\bv_h)  + {\gamma_1}(\nabla \cdot \bu_h^{n+1},\nabla \cdot \bv_h) - (p_h^{n},\nabla \cdot \bv_h) = (\bff, \bv_h), \label{eq:AH1} \\[0.1cm]
\lefteqn{(p_h^{n+1}-p_h^n,q_h) + {\gamma_1}(\nabla \cdot \bu_h^{n+1},q_h) = 0,} \label{eq:AH2} \\[0.1cm]
\lefteqn{\frac{1}{\rho_2}\left( \bJ_h^{n+1}-\bJ_h^n, \bK_h \right)_\bB + (\bJ_h^{n+1},\bK_h) +
{\gamma_2}(\nabla \cdot \bJ_h^{n+1},\nabla \cdot \bK_h)  } \nonumber\\
&&-
(\phi_h^{n},\nabla \cdot \bK_h)+ (\bK_h \times \bB,\bu_h^{n+1}) 
 = (\bf g, \bK_h),\label{eq:AH3} \\[0.1cm]
\lefteqn{ (\phi_h^{n+1}-\phi_h^n, \psi_h) + {\gamma_2}(\nabla \cdot \bJ_h^{n+1},\psi_h)= 0. } \label{eq:AH4}
\end{eqnarray}
\end{algorithm}

\begin{remark}
    For simplicity, the initial conditions are chosen as zero, but any other initialization with $\nabla \cdot \bu^0 = 0$ and $\nabla \cdot \bJ^0 = 0$ can be utilized.
\end{remark}
\begin{remark}
As can be seen, all four equations are decoupled in the new approach. Additionally, the equations \eqref{eq:AH2}-\eqref{eq:AH4} give rise to linear systems with symmetric and positive-definite coefficient matrices. \end{remark}
\begin{remark}
    Note that in 2D, we have $\bB = (0,0,b_0(x,y))$ with $b_0(x,y) \neq 0$, so that
\begin{equation*}
    \|\bJ\|_{\bB} = \sqrt{\int\limits_\Omega |b_0|^2 \bJ \cdot \bJ dx},
\end{equation*}
which in fact is a norm. In 3D, however, $ \|\bJ\|_{\bB} $ is only a semi-norm. One could also consider the AH scheme defined for
$$ (\bJ,\bK)_{\bB}:=\left( \bJ \times \bB,\bK \times \bB \right) + \left( \bJ \cdot \bB,\bK \cdot \bB \right), $$
which, owing to the Binet–Cauchy identity, generates a norm 
\begin{equation*}
    \|\bJ\|_{\bB} = \sqrt{\int\limits_\Omega |\bB|^2 \bJ \cdot \bJ dx},
\end{equation*}
similar to the 2D case.
\end{remark}
We recall two simple lemmas from \cite{Calcolo2026} that will be used later.
\begin{lemma}
\label{SeqConv2Zero}
 If $\{a_j\}_{j=1}^\infty, \{b_j\}_{j=1}^\infty, \{c_j\}_{j=1}^\infty, \{d_j\}_{j=1}^\infty$ with $a_j, b_j, c_j, d_j \ge 0$, and there exists $\mu_j$ and $\varepsilon_j$, $j=\overline{1,3}$, such that $0 < \varepsilon_j \le \mu_j$ with
 \begin{equation*}
  \mu_1 a_{j+1} + \mu_2 b_{j+1} + \mu_3 c_{j+1} + d_{j+1} \leq \left(\mu_1-\varepsilon_1 \right) a_j + \left(\mu_2 - \varepsilon_2 \right) b_{j} + \left(\mu_3 - \varepsilon_3 \right) c_{j} + d_{j},
 \end{equation*}
then there exists $D \ge 0$ such that 
\begin{align*}
 \lim\limits_{j \rightarrow \infty} a_j=\lim\limits_{j \rightarrow \infty} b_j = \lim\limits_{j \rightarrow \infty} c_j =0 \text{ and } \lim\limits_{j \rightarrow \infty} d_j = D.
\end{align*}
\end{lemma}
\begin{lemma}\label{LemmaContractionSequence}
 If $\{a_j\}_{j=1}^\infty, \{b_j\}_{j=1}^\infty, \{c_j\}_{j=1}^\infty, \{d_j\}_{j=1}^\infty$ with $a_j, b_j, c_j, d_j \ge 0$, and there exists $\mu_j, j=\overline{1,4}$ and $\varepsilon_j, j=\overline{1,3}$, such that $0<\varepsilon_j \le \mu_j$ for $j=\overline{1,3}$, with
 \begin{equation}
  \mu_1 a_{j+1} + \mu_2 b_{j+1} + \mu_3 c_{j+1} + \mu_4 d_{j+1} \leq \left(\mu_1-\varepsilon_1 \right) a_j + \left(\mu_2 - \varepsilon_2 \right) b_{j} + \left(\mu_3 - \varepsilon_3 \right) c_{j} + \mu_4 d_j,
  \label{eq:Contract1}
 \end{equation}
 and 
 \begin{equation}
 d_{j+1} \le \tau_1 a_{j+1} + \tau_2 a_{j} + \beta b_{j+1} + {\gamma_1} c_{j+1} + {\gamma_2} c_j \text{ for some positive } \tau_1, \tau_2, \beta, {\gamma_1}, {\gamma_2},
 \label{eq:Contract2}
 \end{equation}
then there is a linear combination of $a_{j}, b_{j}, c_{j}$, $d_j$ that converges towards $0$ geometrically.
\end{lemma}
\begin{proof}
Choose arbitrary $\xi>0$, combine \eqref{eq:Contract1} and \eqref{eq:Contract2} to obtain 
\begin{equation}
  \label{eq:Contract3}
 \begin{aligned}
  (\mu_1-\tau_1 \xi) a_{n+1} & + (\mu_2-\beta \xi) b_{n+1} + (\mu_3 - {\gamma_1}  \xi )c_{n+1} + (\mu_4 + \xi) d_{n+1} \\
  & \leq \left(\mu_1-\varepsilon_1 +\tau_2\xi \right) a_n + \left(\mu_2 - \varepsilon_2 \right) b_{n} + \left(\mu_3 - \varepsilon_3 + {\gamma_2} \xi \right)c_{n} + \mu_4 d_n.
 \end{aligned}
\end{equation}
The first set of conditions on $\xi$ is that $\xi < \min\left\lbrace \dfrac{\mu_1}{\tau_1},\dfrac{\mu_2}{\beta}, \dfrac{\mu_3}{{\gamma_1}}\right\rbrace$.
Moreover, we also assume that 
\begin{equation}
  \label{eq:Contract4}
  \dfrac{\mu_3-\varepsilon_3 + {\gamma_2} \xi}{\mu_3-{\gamma_1} \xi} < \dfrac{\mu_4}{\mu_4 +\xi} \Longleftrightarrow 
F_1(\xi): = {\gamma_2} \xi^2 + \xi(\mu_3-\varepsilon_3 + {\gamma_1} \mu_4 + {\gamma_2} \mu_4) - \mu_4 \varepsilon_3 < 0.
\end{equation}
Noting that $F_1(\xi)$ has two real zeros $\xi_{-1}< 0 < \xi_1$, and requiring
$$\xi < \min\left\lbrace \dfrac{\mu_1}{\tau_1},\dfrac{\mu_2}{\beta}, \dfrac{\mu_3}{{\gamma_1}}, \xi_1 \right\rbrace,$$
we get that 
\begin{equation}
  \label{eq:Contract5}
 \begin{aligned}
  (\mu_1-\tau_1 \xi) a_{n+1} & + (\mu_2-\beta \xi) b_{n+1} + (\mu_3 - {\gamma_1} \xi )c_{n+1} + (\mu_4 + \xi) d_{n+1} \\
  & \leq \left(\mu_1-\varepsilon_1 +\tau_2\xi \right) a_n + \left(\mu_2 - \varepsilon_2 \right) b_{n} + \frac{\mu_4}{\mu_4+\xi} \left[ \left(\mu_3 - {\gamma_1} \xi \right)c_{n} + (\mu_4 + \xi) d_n \right].
 \end{aligned}
\end{equation}
We impose two additional requirements on $\xi$. The first one is 
\begin{equation}
  \label{eq:Contract6}
  \dfrac{\mu_2-\varepsilon_2}{\mu_2-\beta\xi} < \dfrac{\mu_4}{\mu_4 +\xi} \Longleftrightarrow 
 \xi < \dfrac{\mu_4 \varepsilon_2}{\mu_2-\varepsilon_2 + \beta \mu_4}.
\end{equation}
Thanks to \eqref{eq:Contract6} and \eqref{eq:Contract5}, we have 
\begin{equation}
  \label{eq:Contract7}
 \begin{aligned}
  (\mu_1-\tau_1 \xi) a_{n+1} & + (\mu_2-\beta \xi) b_{n+1} + (\mu_3 - {\gamma_1} \xi )c_{n+1} + (\mu_4 + \xi) d_{n+1} \\
  & \leq \left(\mu_1-\varepsilon_1 +\tau_2\xi \right) a_n + \frac{\mu_4}{\mu_4+\xi} \left[ \left(\mu_2 - \beta \xi \right) b_{n} + \left(\mu_3 - {\gamma_1} \xi \right)c_{n} + (\mu_4 + \xi) d_n \right].
 \end{aligned}
\end{equation}
The second requirement on $\xi$ is that 
\begin{equation}
  \label{eq:Contract8}
  \dfrac{\mu_1-\varepsilon_1+\tau_2 \xi}{\mu_1-\tau_1 \xi} < \dfrac{\mu_4}{\mu_4 +\xi} \Longleftrightarrow 
 F_2(\xi): = \tau_2 \xi^2 + \xi(\mu_1 - \varepsilon_1 + \mu_4 \tau_1 + \mu_4 \tau_2 ) - \mu_4 \varepsilon_1<0.
\end{equation}
Similar to $ F_1(\xi)$, $ F_2(\xi)$ also has two real zeros $\xi_{-2}<0<\xi_2$ and we can select $\xi$ to satisfy
\begin{equation}
\xi < \min\left\lbrace \dfrac{\mu_1}{\tau_1},\dfrac{\mu_2}{\beta}, \dfrac{\mu_3}{{\gamma_1}}, \xi_1, {\frac{\mu_4 \varepsilon_2}{
\mu_2-\varepsilon_2 + \beta \mu_4}}, \xi_2 \right\rbrace. 
\label{eq:Contract9}
\end{equation}
If \eqref{eq:Contract9} holds, we obtain that
\begin{equation}
\psi_n: = (\mu_1-\tau_1 \xi) a_{n} + (\mu_2-\beta \xi) b_{n} + (\mu_3- {\gamma_1} \xi) c_{n} + (\mu_4 + \xi) d_{n}   
\label{eq:Contract10}
\end{equation}
 is a sequence contracting towards $0$ with the contraction coefficient $\lambda := \dfrac{\mu_4}{\mu_4+\xi}$.
\end{proof}
\begin{theorem}[Stability]\label{boundedness} 
Let $(\bu_h, p_h,\bJ_h,\phi_h) \in  (\bX_h,Q_h,\bD_h,S_h) $ be the Finite Element Approximation of $(\bu, p,\bJ,\phi)$ solving \eqref{eq:w1}-\eqref{eq:w4} and $\{\bu_h^n, p_h^n,\bJ_h^n,\phi_h^n\}$ solve Algorithm \ref{algo1}. Let 
\begin{equation}
    \label{eq:Lambda0}
    \Lambda_0: = \frac{1}{\Reynolds}
        - \M \|\nabla \bu_h\| -\dfrac{C_4^2 \|\bu_h\|_{L^4(\Omega)} ^2 }{8 \gamma_1}.
\end{equation}
If $\Lambda_0>0$ and $\rho_1$, $\rho_2$ satisfy
\begin{equation}
        \Lambda_1: = \Lambda_0 - \frac{\rho_1 \M^2 \|\nabla \bu_h\|^2}{2} 
        - \frac{\rho_2 \kappa C_p^2}{2}>0,
        \label{eq:Lambda1}
\end{equation}
then the solution of Algorithm \ref{algo1} is bounded, uniformly in $n$ and $h$, and 
 \begin{equation} 
\bu^{n}_h \xrightarrow{H^1}\bu_h, \, 
\Divergence \bu^{n}_h \xrightarrow{L^2} \Divergence \bu_h, \, 
p^{n}_h \xrightarrow{L^2} p_h, \,
\bJ^{n}_h \xrightarrow{\Hdiv}\bJ_h, \,
\phi^{n}_h \xrightarrow{L^2} \phi_h \, \text{ as } n \rightarrow \infty.
 \end{equation}
\end{theorem}
\begin{proof}
We subtract \eqref{eq:AH1}-\eqref{eq:AH4} from \eqref{eq:w1}-\eqref{eq:w4} to obtain the error equation:
\begin{align}
&\frac{(\nabla(\be_h^{n+1}-\be_h^n),\nabla \bv_h)}{\rho_1} + \Reynolds^{-1} (\nabla \be_h^{n+1}, \nabla \bv_h) + b(\be_h^{n},\bu_h,\bv_h) + 
b(\bu_{h}^{n},\be_h^{n+1},\bv_h)
 \nonumber \\
&  - \kappa(\bZ_h^{n} \times\bB,\bv_h)  +{\gamma_1} (\nabla \cdot \be_h^{n+1},\nabla \cdot \bv_h) - (\delta_h^{n},\nabla \cdot \bv_h) = 0, \label{eq:err1} \\[0.2cm]
& (\delta_h^{n+1}-\delta_h^n,q_h) + {\gamma_1}(\nabla \cdot \be_h^{n+1},q_h) = 0, \label{eq:err2} \\[0.2cm]
& \frac{(\bZ_h^{n+1}-\bZ_h^n,\bK_h)_{\bB}}{\rho_2} + (\bZ_h^{n+1},\bK_h)-(\theta_h^{n},\nabla \cdot \bK_h) + (\bK_h \times \bB,\be_h^{n+1}) 
+{\gamma_2} (\nabla \cdot \bZ_h^{n+1},\nabla \cdot \bK_h) = 0, \label{eq:err3} \\
& (\theta_h^{n+1}-\theta_h^n, \psi_h) + {\gamma_2}(\nabla \cdot \bZ_h^{n+1},\psi_h)= 0, \label{eq:err4}
\end{align}
where $\be_h^n=\bu_h-\bu_h^n$, $\delta_h^n=p_h-p_h^n$, $\bZ_h^n=\bJ_h-\bJ_h^n$, $\theta_h^n=\bphi_h-\bphi_h^n$. 
Next, letting $(\bv_h,q_h,\bK_h,\psi_h) = (\be_h^{n+1},\delta_h^{n}, \kappa\bZ^{n+1}_h, \kappa\theta^{n}_h)$ in \eqref{eq:err1}-\eqref{eq:err4} and owing to $ b(\bu^n_h, \be_h^{n+1}, \be_h^{n+1})=0$ we obtain:
\begin{align}
&\frac{\|\nabla \be_h^{n+1}\|^2-\|\nabla \be_h^{n}\|^2 + \|\nabla(\be_h^{n+1}-\be_h^n)\|^2}{2\rho_1} + \frac{\| \nabla \be_h^{n+1}\|^2}{\Reynolds} + {\gamma_1} \|\nabla \cdot \be_h^{n+1}\|^2 - (\delta_h^{n},\nabla \cdot \be_h^{n+1}) \nonumber \\
&\quad = - b(\be_h^{n},\bu_h,\be^{n+1}_h)  + \kappa(\bZ_h^{n} \times\bB,\be_h^{n+1} ), \label{eq:err5} \\[0.2cm]
& \frac{ \|\delta_h^{n+1}\|^2 - \| \delta_h^n \|^2 -
\|\delta_h^{n+1}-\delta_h^{n}\|^2  }{2{\gamma_1}} +  (\nabla \cdot \be_h^{n+1}, \delta_h^{n}) = 0, \label{eq:err6} \\
& \frac{\kappa \Big(\|\bZ_h^{n+1} \|^2_{\bB} - \|\bZ_h^n\|^2_{\bB} + \|\bZ_h^{n+1}-\bZ_h^n\|^2_{\bB}\Big)}{2\rho_2}  + \kappa\| \bZ_h^{n+1} \|^2 + \kappa {\gamma_2} \| \nabla \cdot \bZ_h^{n+1} \|^2 - \kappa(\theta_h^{n},\nabla \cdot \bZ_h^{n+1}) \notag \\
&\quad = - \kappa(\bZ_h^{n+1} \times \bB,\be_h^{n+1}), \label{eq:err7} \\[0.2cm]
& \frac{\kappa (\|\theta_h^{n+1}\|^2 - \| \theta_h^n \|^2 - 
\|\theta_h^{n+1}-\theta_h^{n}\|^2 )}{2{\gamma_2}} +  \kappa(\nabla \cdot \bZ_h^{n+1}, \theta_h^{n}) = 0. \label{eq:err8}
\end{align}
Next, we add the equations {\eqref{eq:err5}-\eqref{eq:err8}} and bound the terms in the RHS. Applying Cauchy-Schwarz, Poincar\'e and Young's inequalities,
the contributions of the Lorentz forcing terms from \eqref{eq:err5} and \eqref{eq:err7} can be expressed as:
\begin{equation}
    \begin{aligned}
    \kappa (\bZ_h^{n} \times\bB,\be_h^{n+1} )
        - \kappa (\bZ_h^{n+1} \times \bB,\be_h^{n+1})
        & = \kappa\left( (\bZ_h^{n}-\bZ_h^{n+1}) \times\bB,\be_h^{n+1} \right)  \\
        &\leq\kappa C_p \|(\bZ_h^{n}-\bZ_h^{n+1}) \times\bB\|\| \nabla \be_h^{n+1}\|\\
        & \le \frac{\kappa}{2 \rho_2}\|\bZ_h^{n+1}-\bZ_h^{n}\|_{\bB}^2 
       + \frac{\rho_2 \kappa C_p^2}{2}  \| \nabla \be_h^{n+1}\|^2. 
    \end{aligned}
    \label{eq:err9}
\end{equation}
 where $C_p$ is Poincar\'e-Friedrichs’s constant. Using Lemma \ref{lem:trilinear}, the contribution of the convective term can be bounded as
\begin{equation}
\label{eq:err10}
    \begin{aligned}
    |b(\be_h^n,\bu_h,\be_h^{n+1})| & \leq \M \|\nabla \bu_h \| \|\nabla \be_h^n\| \|\nabla \be_h^{n+1}\| + \dfrac{C_4}{2}\|\nabla \cdot \be_h^n\| \|\bu_h\|_{L^4(\Omega)}\|\nabla \be_h^{n+1}\|\\
    & \leq \M \|\nabla \bu_h\|\left( \|\nabla \be_h^{n+1}\|^2 + \|\nabla \be_h^{n+1}\| \| \nabla (\be_h^{n+1}-\be_h^{n})\|\right) \\
     & + \dfrac{\gamma_1}{2}\|\nabla \cdot \be_h^n\|^2+ \dfrac{C_4^2 \|\bu_h\|_{L^4(\Omega)}^2 }{8 \gamma_1} \|\nabla \be_h^{n+1}\|^2\\
    & \leq \left[\M \|\nabla \bu_h\| +\dfrac{C_4^2 \|\bu_h\|_{L^4(\Omega)}^2 }{8 \gamma_1}+ \frac{\rho_1 \M^2 \|\nabla \bu_h\|^2}{2}  \right] \|\nabla \be_h^{n+1}\|^2 +
    \dfrac{\|\nabla (\be_h^{n+1}-\be_h^n)\|^2}{2\rho_1} \\
    & + \dfrac{\gamma_1}{2}\|\nabla \cdot \be_h^n\|^2. 
\end{aligned}
\end{equation}
Setting $q_h=\delta_h^{n+1}-\delta_h^n$ and $\psi_h=\theta_h^{n+1}-\theta_h^n$ in \eqref{eq:err2} and \eqref{eq:err4} produces
\begin{eqnarray}
        \|\delta_h^{n+1}-\delta_h^n\| \le \gamma_1 \|\nabla\cdot \be_h^{n+1} \|,\nonumber\\
        \|\theta_h^{n+1}-\theta_h^n\| \le \gamma_2 \|\nabla\cdot \bZ_h^{n+1} \|.\label{eq:er11}
\end{eqnarray}
Combining \eqref{eq:err5}-\eqref{eq:er11} gives
\begin{eqnarray}
\label{eq:err11}
\lefteqn{\left[\frac{1}{2 \rho_1} +  \Lambda_0 - \frac{\rho_1 \M^2 \|\nabla \bu_h\|^2}{2} 
- \frac{\rho_2 \kappa C_p^2}{2} \right] \|\nabla \be_h^{n+1}\|^2} \nonumber \\
       &&
+ \dfrac{\gamma_1}{2} \|\nabla \cdot \be_h^{n+1}\|^2+ \dfrac{\gamma_1}{2} \left(\|\nabla \cdot \be_h^{n+1}\|^2-\dfrac{1}{\gamma_1^2}\|\delta_h^{n+1}-\delta_h^{n}\|^2\right)\nonumber \\
       &&
+ \kappa \left(\| \bZ_h^{n+1} \|^2 + \frac{\|\bZ_n^{n+1}\|_\bB^2}{2 \rho_2} \right) + \frac{1}{2{\gamma_1} }\| \delta^{n+1}_h \|^2 + \frac{\kappa}{2{\gamma_2}} 
\|\theta^{n+1}_h \|^2 \nonumber\\
       &&
+ \dfrac{\kappa \gamma_2}{2} \| \nabla \cdot \bZ_h^{n+1} \|^2 + \dfrac{\kappa \gamma_2}{2}  \left(\| \nabla \cdot \bZ_h^{n+1} \|^2-\dfrac{1}{\gamma_2^2} \|\theta_h^{n+1}-\theta_h^{n}\|^2\right)\nonumber\\
    &\le& \frac{1}{2\rho_1}\|\nabla \be_h^{n}\|^2 + {\dfrac{\gamma_1}{2}\|\nabla \cdot \be_h^n\|^2}+  \frac{\kappa}{2\rho_2} \|\bZ_h^{n} \|_\bB^2 + \frac{1}{{2{\gamma_1} }}\| \delta^{n}_h \|^2 + \frac{\kappa}{{2{\gamma_2} }}\| \theta^{n}_h \|^2,
\end{eqnarray}
where $\Lambda_0$ is given in \eqref{eq:Lambda0}. Note that choosing $\rho_1 $ and $\rho_2$ as per assumption \eqref{eq:Lambda1} ensures that the coefficient of $\|\nabla \be_h^{n+1} \|^2$ is greater than that of $\|\nabla \be_h^{n} \|^2$.
Applying \eqref{eq:er11}, the nonnegative terms in \eqref{eq:err11} can be dropped, which yields 
\begin{equation}
\label{eq:err15}
\begin{aligned}   
\left[\frac{1}{2 \rho_1} + \Lambda_1 \right] \|\nabla \be_h^{n+1}\|^2 &+ \kappa \left(\| \bZ_h^{n+1} \|^2 + \frac{\|\bZ_n^{n+1}\|_\bB^2}{2 \rho_2} \right)
 + \dfrac{\gamma_1}{2} \|\nabla \cdot \be_h^{n+1}\|^2+ \frac{1}{2{\gamma_1} }\| \delta^{n+1}_h \|^2 \\
& + \frac{\kappa}{2{\gamma_2}}\| \theta^{n+1}_h \|^2 + \dfrac{\kappa \gamma_2}{2} \| \nabla \cdot \bZ_h^{n+1} \|^2 \\
& \le \frac{1}{2\rho_1}\|\nabla \be_h^{n}\|^2 
+ {\dfrac{\gamma_1}{2}\|\nabla \cdot \be_h^n\|^2}+ \kappa \frac{\|\bZ_h^{n}\|_\bB^2}{2 \rho_2} + \frac{1}{{2{\gamma_1} }}\| \delta^{n}_h \|^2 + \frac{\kappa}{{2{\gamma_2} }}\| \theta^{n}_h \|^2. 
\end{aligned}
\end{equation}
Then uniform boundedness follows by induction on $n$ and the triangle inequality. Moreover, using the Lemma \ref{SeqConv2Zero} with 
$$
a_n = \|\nabla \be_h^{n}\|^2, \, 
b_n = \kappa\| \bZ_h^{n}\|^2, \, c_n=  \frac{\kappa \gamma_2}{2}\| \nabla \cdot \bZ_h^{n} \|^2, $$ 
$$d_n = {\frac{\gamma_1}{2}\| \nabla \cdot \be_h^{n} \|^2 }+\kappa \frac{\|\bZ_h^{n}\|_\bB^2}{2 \rho_2} + \frac{1}{2{\gamma_1} }\| \delta^{n}_h \|^2 + \frac{\kappa}{2{\gamma_2}}\| \theta^{n}_h \|^2, $$
and
$$
\mu_1=\frac{1}{2\rho_1} + \Lambda_1, \, 
\mu_2=1, \, \mu_3=1, \, \varepsilon_1= \Lambda_1, \, \varepsilon_2=1, \varepsilon_3=1,$$
it follows that 
\begin{equation}
    \label{eq:err16}
    \lim \limits_{n \rightarrow \infty} \|\nabla \be_h^{n}\|^2=\lim \limits_{n \rightarrow \infty}  \|\nabla \cdot \be_h^{n}\|^2 =\lim \limits_{n \rightarrow \infty} \| \bZ_h^{n} \|^2=\lim \limits_{n \rightarrow \infty}  \|\nabla \cdot \bZ_h^{n}\|^2=
    \lim \limits_{n \rightarrow \infty} \| \bZ_h^{n} \|_\bB^2 = 0.
\end{equation}
The last result is true thanks to 
$$\|\bZ\|_\bB \le \|\bB\|_{\bL^\infty(\Omega)} \|\bZ\|.$$
Next we show the convergence of pressure approximations. To this end, apply the inf-sup conditions 
in \eqref{eq:err3} to get 
\begin{equation}    
\label{eq:err17}
\begin{aligned}
\beta_m \| \theta_h^{n} \| &\le \rho_2^{-1} \|\bB\|_{\bL^\infty(\Omega)} \|\bZ_h^{n+1}-\bZ_h^{n} \|_{\bB} +\| \bZ_h^{n+1} \|+ \gamma_2\| \nabla \cdot \bZ_h^{n+1} \| + C_p \| \bB \|_{\bL^\infty(\Omega)} \|\nabla \be_h^{n+1} \| \\
& \leq  \left(1+ \rho_2^{-1}\|\bB\|_{\bL^\infty(\Omega)}^2\right)\|\bZ_h^{n+1}\|+ \rho_2^{-1}\|\bB\|_{\bL^\infty(\Omega)}^2\|\bZ_h^{n}\| \\ 
& + \gamma_2\| \nabla \cdot \bZ_h^{n+1} \| + C_p \| \bB \|_{\bL^\infty(\Omega)} \|\nabla \be_h^{n+1} \|,
\end{aligned}
\end{equation}
and then let $n \rightarrow \infty$ to get the desired result. The convergence of the electric potential $\phi^n_h$ can be established similarly.
\end{proof}

 \begin{remark}
From the "small data" condition \eqref{eq:Lambda1} it follows that ${\gamma_1} \gg 1$ is more favorable for the convergence of Algorithm \ref{algo1}, which also deteriorates the conditioning number of the associated linear system. 
 \end{remark}
%
\begin{theorem}[Contractivity]\label{contractivity} 
Let $(\be^n_h, \delta^n_h,\bZ^n_h,\theta^n_h) $ be as in the Theorem \ref{boundedness}. If $\Lambda_0 > 0 $ and $\rho_1, \rho_2$ satisfy \eqref{eq:Lambda1}, then there exists $\mu_j>0$, $j=\overline{1,4}$ and $\lambda \in (0,1)$ such that
 \begin{equation} 
 \begin{aligned}
 \psi_n: & = \mu_1 \| \nabla \be^{n}_h \|^2
+ \mu_2 \| \bZ_h^{n} \|^2+ \mu_3 
\| \nabla \cdot \bZ_h^{n} \|^2   \\
& + \mu_4 \left({\frac{\gamma_1}{2}\| \nabla \cdot \be_h^{n}  \|^2}+\kappa \frac{\| \bZ^n_h\|_\bB^2}{2\rho_2} + \dfrac{1}{2{\gamma_1}}\| \delta^{n}_h \|^2 + \frac{{\kappa} }{2{\gamma_2}}\| \theta^{n}_h \|^2\right),
 \end{aligned}
\label{eq:ContractionThm1}
 \end{equation}
satisfies 
 \begin{equation}
     \psi_{n+1} < \lambda \psi_n.
     \label{eq:ContractionThm2}
 \end{equation}
\end{theorem}
\begin{proof}
By using the notation from the proof of the last Theorem, and setting $\mu_4=1$,
we see that \eqref{eq:err15} is equivalent to \eqref{eq:Contract1}. Moreover, thanks to the inf-sup conditions, we get
\begin{equation}
\label{eq:ContractionThm4}
    \begin{aligned}
\beta_s \|\delta_h^{n}\| & \le \left( \rho^{-1}_1 + \frac{1}{\Reynolds} + \M \|\nabla \bu_h^n \| \right)  \| \nabla \be_h^{n+1} \| 
   + \left( \rho^{-1}_1 + \M \|\nabla \bu_h \| \right)  \| \nabla \be_h^{n} \| \\
   & + \gamma_1 \| \Divergence \be_h^{n+1} \| + \kappa C_p \| \bB \|_{\bL^\infty(\Omega)}  \|\bZ^n_h\|,\\[0.2cm]
   \beta_m \| \theta_h^{n}\| & \le  \left(1+ \rho_2^{-1}\|\bB\|_{\bL^\infty(\Omega)}^2\right)\|\bZ_h^{n+1}\|+ \rho_2^{-1}\|\bB\|_{\bL^\infty(\Omega)}^2\|\bZ_h^{n}\| \\ 
   & + \gamma_2\| \nabla \cdot \bZ_h^{n+1} \| + C_p \| \bB \|_{\bL^\infty(\Omega)} \|\nabla \be_h^{n+1} \|,
    \end{aligned}
\end{equation}
which together with the uniform boundedness of the solution implies \eqref{eq:Contract2}. The proof is then completed by application of Lemma \ref{LemmaContractionSequence}. 
\end{proof}
\begin{corollary}\label{Errbound} 
Assume that $(\bu_h,p_h,\bJ_h,\phi_h)$ and $(\bu_h^n,p_h^n,\bJ_h^n,\phi_h^n)$ solve  \eqref{eq:w1}-\eqref{eq:w4} and \eqref{eq:AH1}-\eqref{eq:AH4}, respectively. Then the error estimate satisfies
\begin{eqnarray}
\|\nabla(\bu_h- \bu_h^{n+1})\|^2 &\leq&\zeta  \lambda^{n+1}\left(\frac{1}{2 \rho_1} + \Lambda_1 \right)^{-1}\dfrac{\|\bf F\|_*^2}{C_{\min}^2}, \label{eq:erresma1}\\ 
        \|p_h-p_h^{n+1}\|^2 &\leq &2\zeta \lambda^{n+1} {{\gamma_1}}\dfrac{\|\bf F\|_*^2}{C_{\min}^2},  \label{eq:erresm2} \\
        \|\bJ_h-\bJ_h^{n+1}\|_{\Hdiv}^2&\leq& \zeta \lambda^{n+1}  \dfrac{\|\bf F\|_*^2}{C_{\min}^2}, \label{eq:erresm3}\\
        \|\phi_h-\phi_h^{n+1}\|^2  &\leq& 2\zeta \lambda^{n+1} {\gamma_2} \kappa^{-1}\dfrac{\|\bf F\|_*^2}{C_{\min}^2},\label{eq:erresm4}
   \end{eqnarray}
   where
 \begin{align}
     \zeta=&
 \frac{1}{2 \rho_1} + \Lambda_1+\dfrac{\gamma_1}{2}+1
+\frac{ \kappa\|\bB\|_{\bL^\infty(\Omega)}^2}{2 \rho_2 } + \dfrac{\beta_s^{-2}}{{2\gamma_1}}\Big(\Reynolds^{-1}+\kappa C_p \| \bB \|_{\bL^\infty(\Omega)}+(\sigma+1)C_{\min}\Big)^2\nonumber\\
&+\dfrac{ \kappa \beta_m^{-2} }{2\gamma_2} \left(1+C_p \| \bB \|_{\bL^\infty(\Omega)} +C_{\min}\right)^2 . \label{eq:zeta}
 \end{align}
\end{corollary}
\begin{proof} It follows from the previous theorem 
that \begin{equation}
\psi_{n+1} < \lambda \psi_n,
   \label{eq:ContractionCor1}
    \end{equation}
    where
 \begin{equation}
     \psi_n: = \mu_1 \| \nabla \be^{n}_h \|^2
+ \mu_2 \| \bZ_h^{n} \|^2+ \mu_3 
\| \nabla \cdot \bZ_h^{n} \|^2  
+ \mu_4 \left(\dfrac{\gamma_1}{2}\| \nabla \cdot \be_h^{n} \|^2+\kappa \frac{\| \bZ^n_h\|_\bB^2}{2\rho_2} + \dfrac{1}{2{\gamma_1}}\| \delta^{n}_h \|^2 + \frac{{\kappa} }{2{\gamma_2}}\| \theta^{n}_h \|^2\right),
 \end{equation}
with the same parameters as in Theorem \eqref{contractivity}. 
Utilizing induction results in
 \begin{equation}
     \psi_{n+1} \leq \lambda^{n+1}\psi_0, \label{eq:ContractionCor2}
 \end{equation}
 where 
 $$ 
\psi_0 = \mu_1 \| \nabla \be^{0}_h \|^2
+ \mu_2 \| \bZ_h^{0} \|^2+\mu_3 
\| \nabla \, \cdot \, \bZ_h^{0} \|^2
+ \mu_4 \left(\dfrac{\gamma_1}{2}\| \nabla \, \cdot \, \be_h^{0} \|^2+\kappa \dfrac{\| \bZ^0_h\|_\bB^2}{2\rho_2}+\dfrac{1}{2{\gamma_1}}\| \delta^{0}_h \|^2 + \dfrac{{\kappa} }{2{\gamma_2}}\| \theta^{0}_h \|^2\right).$$ 
Since $(\bu_h^0,p_h^0,\bJ_h^0,\phi_h^0) = (0,0,0,0)$, then one has
\begin{equation}
        \be^{0}_h=\bu_h, \ 
         \bZ_h^0=\bJ_h,\ 
         \delta_h^0=p_h, \ 
         \theta_h^0= \phi_h. \label{eq:err0}
          \end{equation}
Applying the inf-sup condition leads to
\begin{eqnarray}
    \|p_h\| &\leq& \beta_s^{-1} \Big(\Reynolds^{-1}\|\nabla \bu_h\| +\M \|\nabla \bu_h \|^2 +\kappa C_p \| \bB \|_{\bL^\infty(\Omega)} \|\bJ_h\|+\|\bf f\|_{-1}\Big)\nonumber\\
    &\leq&\beta_s^{-1} \left(\Reynolds^{-1}+\kappa C_p \| \bB \|_{\bL^\infty(\Omega)}+(\sigma+1)C_{\min}\right)\dfrac{\|\bf F\|_*}{C_{\min}},
    \label{eq:ph}\\[0.2cm]
    \|\phi_h\| &\leq& \beta_m^{-1}\Big(\|\bJ_h\|+C_p \| \bB \|_{\bL^\infty(\Omega)} \|\nabla \bu_h\|+\|\bg\|\Big)\nonumber\\
    &\leq&\beta_m^{-1}\left(1+C_p \| \bB \|_{\bL^\infty(\Omega)}  +C_{\min}\right)\dfrac{\|\bf F\|_*}{C_{\min}}.\label{eq:phi}
    \end{eqnarray}
By using \eqref{eq:uJfembound}, and inserting \eqref{eq:ph} and \eqref{eq:phi} into \eqref{eq:ContractionCor2}, we get
\begin{align}
\mu_1 \| \nabla \be^{n+1}_h \|^2
&+ \mu_2 \| \bZ_h^{n+1} \|^2+\mu_3
\| \nabla \cdot \bZ_h^{n+1} \|^2 \\
& + \mu_4 \left(\dfrac{\gamma_1}{2}\| \nabla \cdot \be_h^{n+1} \|^2+\kappa \frac{\| \bZ^{n+1}_h\|_\bB^2}{2\rho_2} +\dfrac{1}{2{\gamma_1}}\| \delta^{n+1}_h \|^2 + \frac{{\kappa} }{2{\gamma_2}}\| \theta^{n+1}_h \|^2\right)\nonumber\\
\leq & \lambda^{n+1}
\Bigg( \frac{1}{2 \rho_1} + \Lambda_1+\frac{\gamma_1}{2}+1
+\frac{ \kappa\|\bB\|_{\bL^\infty(\Omega)}^2}{2 \rho_2 } + \dfrac{\beta_s^{-2}}{{2\gamma_1}}\Big(\Reynolds^{-1}+\kappa C_p\| \bB \|_{\bL^\infty(\Omega)}+(\sigma+1)C_{\min}\Big)^2\nonumber\\
&+\dfrac{ \kappa \beta_m^{-2} }{2\gamma_2} \left(1+C_p \| \bB \|_{\bL^\infty(\Omega)} +C_{\min}\right)^2\Bigg)\dfrac{\|\bf F\|_*^2}{C_{\min}^2},
\end{align}
which yields the results \eqref{eq:erresma1}-\eqref{eq:erresm4}.
 \end{proof}
\begin{theorem}[Error bound]
Suppose $(\bu,p,\bJ,\phi)$  and $(\bu_h^n,p_h^n,\bJ_h^n,\phi_h^n)$ be solutions of \eqref{eq:ef1}-\eqref{eq:ef4} and \eqref{eq:AH1}-\eqref{eq:AH4}. If 
\begin{equation*}
\bu \in \bX \cap \bH^{1+\gamma}(\Omega), \, p \in Q \cap H^{\gamma}(\Omega), \, \bJ \in \bD \cap \bH^{k}(\Omega) \text{ with } \Divergence \bJ \in \bH^{k}(\Omega), \, \text{ and } \phi \in S \cap H^{k}(\Omega),
\end{equation*}
then
\begin{align*}
\|(\bu-\bu_h^{n+1},\bJ-\bJ_h^{n+1})\|_1+\|p-p_h^{n+1}\|
   &\leq C h^{\min\{\gamma,k\}}(\|\bu\|_{1+\gamma}+\|\bJ\|_{k}+\|\Divergence \bJ\|_{k}+\|p\|_{\gamma}) \nonumber\\
            &+ \sqrt{\lambda^{n+1}\zeta}  \left(\sqrt{\left(\frac{1}{2 \rho_1} + \Lambda_1 \right)^{-1}+1}+\sqrt{{2{\gamma_1}}}\right)\dfrac{\|\bf F\|_*}{C_{\min}},
           \nonumber\\[0.2cm]
\|\phi-\phi_h^{n+1}\| &\leq C h^{\min\{\gamma,k\}}\left(\|\bu\|_{1+\gamma}+\|\bJ\|_{k}+\|\Divergence \bJ\|_{k}+\|p\|_{\gamma}+\|\phi\|_{k} \right) \\
& \quad + \sqrt{2\lambda^{n+1}\zeta \gamma_2 \kappa^{-1} }\dfrac{\|\bf F\|_*}{C_{\min}},\nonumber
 \end{align*}
where $\zeta$ is defined by \eqref{eq:zeta}.
    \end{theorem}
    \begin{proof}
      Using the triangle inequality yields 
      \begin{eqnarray}
   \lefteqn{ \| (\bu-\bu_h^{n+1},\bJ-\bJ_h^{n+1}) \|_1+\|p-p_h^{n+1} \| } \nonumber \\
            &\leq &\|(\bu-\bu_h,\bJ-\bJ_h)\|_1+\|p-p_h \|+  \|(\bu_h-\bu_h^{n+1},\bJ_h-\bJ_h^{n+1})\|_1+\|p_h-p_h^{n+1} \|, \nonumber 
               \end{eqnarray}
        and 
     \begin{equation*}
           \|\phi-\phi_h^{n+1}\| \leq \|\phi-\phi_h\|+\|\phi_h-\phi_h^{n+1}\|.
     \end{equation*}           
The stated result follows from the combination of Theorem \ref{exactbd} and Corollary \ref{Errbound}.    
    \end{proof}

\section{Numerical Experiments}
\label{sec:Numerics}
To test the effectiveness of our novel Algorithm \ref{algo1} and to validate the theoretical results, we present several numerical examples in this section. The first example is a convergence test using a smooth exact solution. The second example, which features singular solutions, aims to assess the performance and convergence rates of the numerical solutions in a non-convex L-shaped domain. Then we test our Algorithm on two-- and three--dimensional lid-driven cavity flows. For all examples, the FreeFem++ finite element library \cite{hec12} is employed for implementation. $(\bP^d_2,P_1)$ Taylor-Hood elements are used for velocity and pressure, and first-order Raviart-Thomas elements $\mathrm{RT}_1$ and discontinuous finite element spaces $P_{1,dc}$ are applied for the current density and electric potential in all of the 2D tests, respectively. The 3D lid-driven cavity problem was approximated via ($\mathrm{RT}_0$,$P_0$) pair. The UMFPACK routine is used to solve the linear systems arising from the discrete algebraic equations in 2D examples. The iterative tolerance is set to $\epsilon=10^{-{ 6}} $ for all the problems with the stopping criteria given by
\begin{equation}
       \frac{\|p_h^n-p_h^{n-1}\|}{\|p_h^n\|} \le \epsilon.
\end{equation}
In all computations, we set $\rho_1=\rho_2=\rho$ and  $\gamma_1=\gamma_2=\gamma$. For the 2D lid-driven cavity test, we used $\gamma=100$, and for the 3D lid-driven cavity problem, we used $\gamma=1$ to keep the condition numbers reasonable. As dictated by the inequality \eqref{eq:Lambda1}, $\rho$ must be chosen small enough, and it should decrease as $\kappa$ and $\Reynolds$ increase. On the other hand, an extremely small value of $\rho$ will deteriorate the conditioning of the associated linear system and should be avoided. In our tests for the lid-driven cavity problem, we empirically tuned the value of $\rho$ by testing multiple values on a coarser mesh and opted for the one that gave the fastest convergence.
\subsection{Convergence for the problem with smooth solution}
\label{subsec:Convergence}
We assess the accuracy and convergence behavior by first constructing a manufactured solution on the unit square as:
$\bB = (0, 0, 1)^T$, $\Reynolds =1$, $\kappa=1$, and choose the right-hand side functions $\bff$
and $\bfg$ such that the exact solutions are
\begin{align*}
    u_1=4 x^2(x-1)^2y(y-1)(2y-1), && 
    u_2=-4 y^2 (y-1)^2 x(x-1)(2x-1), \nonumber\\
    p=(2x-1)(2y-1),\nonumber &&\phi=x-1/2,\\
    J_1=4 \sin(\pi x)\cos(\pi y),&&J_2=-4 \sin(\pi y)\cos(\pi x).
\end{align*}
The corresponding errors and convergence rates are presented in Table \ref{tab:iahprob1}. We can see that the expected convergence rates are achieved for all quantities and the number of iterations remains nearly uniform. 
\begin{table}[h]
    \centering
    \scalebox{0.9}{
\begin{tabular}{|c | c | c | c |c |c |c |c| c |c|c|c|} 
\hline
h&	CPU& iter &	$\|\nabla \be_{\bu,h}^{n+1}\|$ &rate& $\|e_{p,h}^{n+1}\|$ & rate& $\|\bZ_h\|_{\Hdiv}$&rate& $\|e_{\phi,h} \|$&rate&	$\|\nabla \cdot\bJ_h \|$ \\ [0.5ex] 
\hline
               0.177 & 0.339 & 8 & 0.0097 & - & 0.03 & - & 0.061 & - & 0.0008 & - & 4.74E-08 \\ \hline
        0.088 & 0.621 & 4 & 0.002 & 2.23 & 0.0067 & 2.16 & 0.0138 & 2.13 & 6.48E-05 & 3.56 & 2.11E-08 \\ \hline
        0.044 & 2.248 & 3 & 0.0005 & 2.04 & 0.0016 & 2.03 & 0.0033 & 2.05 & 8.07E-06 & 3.00 & 2.76E-09 \\ \hline
        0.022 & 11.846 & 3 & 0.0001 & 1.95 & 0.0004 & 2.02 & 0.0008 & 2.00 & 9.93E-07 & 3.02 & 1.09E-09 \\
        \hline
    \end{tabular} 
    }
    \caption{Errors, convergence rates, CPU times for $\rho=5$ with Algorithm \ref{algo1} for smooth manufactured solution test \ref{subsec:Convergence}}
    \label{tab:iahprob1}
\end{table}
\subsection{Convergence test for the problem with singular solutions}
\label{subsec:ConvergenceLshaped}
Next, we investigate the problem with singular solutions from \cite{L88} to verify that the proposed method can effectively capture these singularities. The Inductionless MHD equations are considered in a non-convex L-shaped domain, \(\Omega = \left( {-0.5, 0.5} \right)^2 \setminus \left( \left[ 0, 0.5 \right) \times \left( -0.5, 0 \right] \right)\). The exact solutions exhibit prominent singularities at the coordinate origin, as it lies at the re-entrant corner of the domain. The system parameters are chosen as $\bB = (0, 0, 1)$ and $\Reynolds = \kappa = 1$, while the source terms and boundary conditions are tailored to match the analytical solution. Using polar coordinates \((r, \theta)\), with \(\omega = \dfrac{3\pi}{2}\), the solution elements for the Navier-Stokes problem, including velocity \(\bu\) and pressure \(p\), are given as
\begin{align*}
  u_1(r,\theta)&=r^\lambda((1+\lambda)\sin(\theta) \Psi(\theta)+\cos(\theta)\Psi'(\theta)),\nonumber\\
    u_2(r,\theta)&=r^\lambda(-(1+\lambda)\cos(\theta) \Psi(\theta)+\sin(\theta)\Psi'(\theta)),\nonumber\\
    p(r,\theta)&=-r^{\lambda-1}((1+\lambda)^2\Psi'(\theta)+\Psi'''(\theta))/{(1-\lambda)},
\end{align*}
with
\begin{equation*}        \Psi(\theta)= \sin((1+\lambda)\theta)\dfrac{\cos(\lambda\omega)}{1+\lambda}-\cos((1+\lambda)\theta)-\sin((1-\lambda)\theta)\dfrac{\cos(\lambda\omega)}{1-\lambda}+\cos((1-\lambda)\theta).
\end{equation*}
The value of the parameter $\lambda$ represents the smallest positive solution of $\sin(\lambda\omega)+\lambda \sin(\omega) = 0$. The current density $\bJ$ and the electric potential $\phi$ are defined as
\begin{equation}
\bJ(r,\theta)=\nabla(r^{2/3}\sin(2\theta/3)),\hspace{1cm}\phi(r,\theta)=0.
\end{equation}
In the case of electric potential, the boundary condition is chosen as $\phi=0 $.

The true solution has a low regularity with $(\bu,p) \in \bH^{1+\lambda}(\Omega)\times H^{\lambda}(\Omega)$, see, e.g., \cite{ubcsingular}. Tables  \ref{table:3a} and \ref{table:4} present the CPU time and errors for numerical solutions of the stationary IMHD. The results are consistent with the results in \cite{zhang}. The plots of the velocity field and the current density are given in Figures \ref{fig:LshapedDomainFigsU}-\ref{fig:LshapedDomainFigsJ}.
 We also used this test to compare Algorithm \ref{algo1} with the AH method of \cite{xia_arrow}, cf. Table \ref{table:6}.

\begin{table}[h!]
\centering
\scalebox{0.9}{
\begin{tabular}{|c|c|c|c|c|c|c|c|c|c|c |c|c|c|} 
\hline
$h$&	CPU &iter&	  $\|\nabla \be_{\bu,h}\|$&	rate&	$\|e_{p,h}\|$ &	rate&	 $\|e_{\phi,h} \|$&	rate&	$\|\bZ_h\|_{\Hdiv}$&	rate&	$\|\nabla\cdot \bZ_h\|$\\ [0.5ex]  \hline
0.25&	3.381&	202&	1.817&	-&	5.674&	-&	0.006&	-&	0.079&	-& 0.0015	\\ \hline
0.125&	6.048&	116&	1.108&	0.714&	1.94&	1.548&	0.004&	0.763&	0.043&	0.872&	0.0011\\ \hline
0.063&	14.053&	94&	0.748&	0.568&	1.264&	0.619&	0.002&	0.884&	0.029&	0.595&	0.0006\\ \hline
0.031&	47.05&	78&	0.564&	0.408&	0.915&	0.465&	0.001&	0.705&	0.016&	0.822&	0.0003 \\ \hline
0.016&	152.925&	56&	0.548&	0.039&	0.832&	0.137&	0.001&	0.35&	0.011&	0.575&	0.0001
	\\ \hline
\end{tabular}}
\caption{ \small Errors, convergence rates, CPU times for $\rho=0.88$
with Algorithm \ref{algo1} for test \ref{subsec:ConvergenceLshaped}}
\label{table:3a}
\end{table}
As can be seen from Table \ref{table:6} our scheme converges faster as $\rho$ increases, whereas Algorithm AH \cite{xia_arrow} fails to achieve stability or convergence for higher values of $\rho$. 
\begin{table}[h!]
\centering
\scalebox{0.9}{
\begin{tabular}{|c | c | c | c |c |c |c |c| c |c|c|c|} 
\hline
$h$&	CPU &	iter&	 $\|\nabla \be_{\bu,h}\|$&	rate&	$\|e_{p,h}\|$ &	rate&	 $\|e_{\phi,h} \|$&	rate&	$\|\bZ_h\|_{\Hdiv}$&	rate&	$\|\nabla \cdot \bZ_h\|$\\  [0.5ex] \hline  
0.25&	2.25&	170&	1.82&	-&	6.6&	-&	0.007&	-&	0.08&	-&	0.0011 \\ \hline
0.125&	4.01&	93&	1.11&	0.71&	2.11&	1.65&	0.004&	0.71&	0.04&	0.86&	0.0009 \\ \hline
0.063&	11.06&	76&	0.75&	0.57&	1.32&	0.68&	0.002&	0.9&	0.03&	0.59&	0.0005 \\ \hline
0.031&	37.02&	64&	0.56&	0.42&	0.93&	0.5&	0.001&	0.75&	0.02&	0.82&	0.0002 \\ \hline
0.016&	169.33&	50&	0.55&	0.04&	0.84&	0.15&	0.001&	0.41&	0.01&	0.57&	0.0001
\\ \hline
\end{tabular}}
\caption{ Errors, convergence rates, CPU times for $\rho=1.25$ 
with Algorithm \ref{algo1} for the $L$-shaped domain test \ref{subsec:ConvergenceLshaped}}
\label{table:4}
\end{table}
\begin{table}[h!]
\centering
\scalebox{0.9}{
\begin{tabular}{|c|c|c|c|c|c|}
    \hline
    \multicolumn{3}{|c|}{$\rho=0.88$}  & \multicolumn{3}{|c|}{$\rho=100$}  \\
    \hline
    $\gamma$ & AH scheme of \cite{xia_arrow} & Algorithm \ref{algo1} &$\gamma$ & AH scheme of \cite{xia_arrow} & Algorithm \ref{algo1} \\
    \hline &&&&&\\ 
    $\dfrac{1}{0.88} $& 64 & 56 &0.01& unstable &19\\[0.2cm]
    \hline
     &&&&&\\ 
      $\dfrac{10}{0.88}$ & 369 & 422&0.1& unstable &122\\ [0.2cm]
    \hline  &&&&&\\ 
      $\dfrac{100}{0.88}$ & 2137 & 2498 &1& unstable &594\\ [0.2cm]
    \hline
\end{tabular}}
\caption{Iterations counts for convergence for the $L$-shaped domain test \ref{subsec:ConvergenceLshaped}}
\label{table:6}
\end{table}
\begin{figure}[H]
\centering
\begin{subfigure}{0.3\textwidth}
    \includegraphics[width=\textwidth]{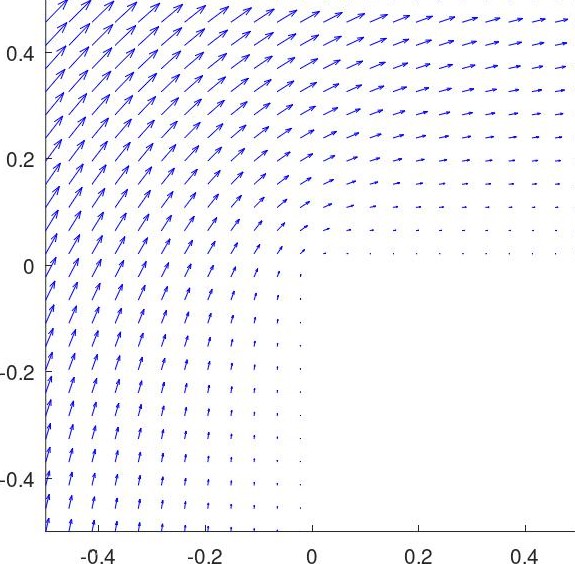}
    \caption{Velocity field $\bu$}
    \label{fig:u_Lshaped}
\end{subfigure}
\begin{subfigure}{0.3\textwidth}
    \includegraphics[width=\textwidth]{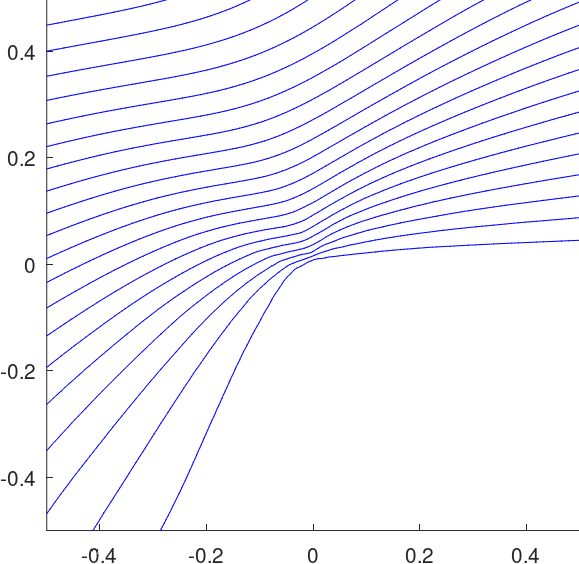}
    \caption{Contours of $u_1$}
\end{subfigure}
\begin{subfigure}{0.3\textwidth}
    \includegraphics[width=\textwidth]{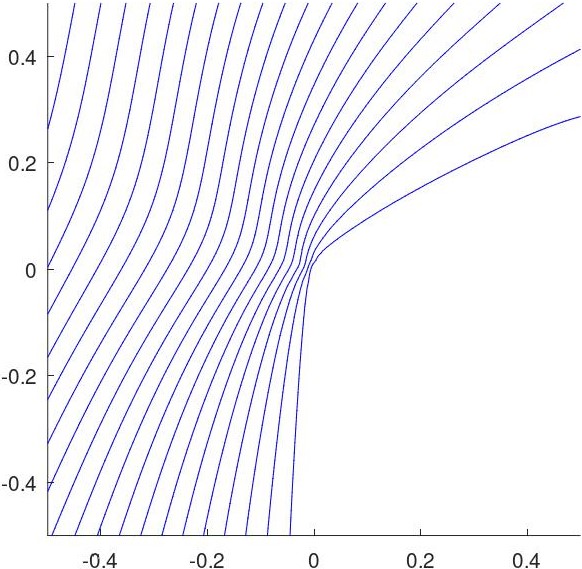}
    \caption{Contours of $u_2$}
\end{subfigure}
\caption{Plots of the velocity field for the $L$-shaped domain test \ref{subsec:ConvergenceLshaped}}
\label{fig:LshapedDomainFigsU}
\end{figure}
\begin{figure}[H]
\centering
\begin{subfigure}{0.3\textwidth}
    \includegraphics[width=\textwidth]{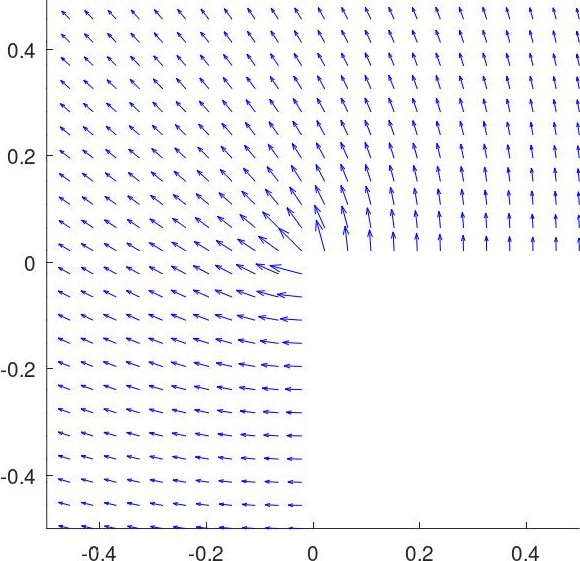}
    \caption{Current density $\bJ$}
    \label{fig:J_Lshaped}
\end{subfigure}
\begin{subfigure}{0.3\textwidth}
    \includegraphics[width=\textwidth]{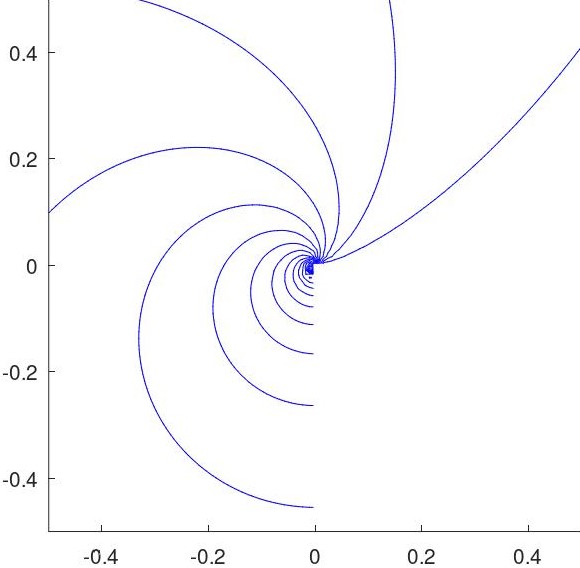}
    \caption{ Contours of $J_1$}
\end{subfigure}
\begin{subfigure}{0.3\textwidth}
    \includegraphics[width=\textwidth]{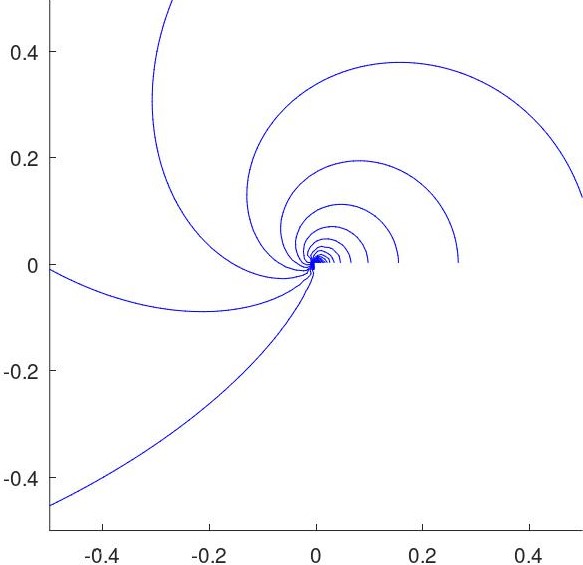}
    \caption{Contours of $J_2$}
\end{subfigure}        
\caption{Plots of current density for the $L$-shaped domain test \ref{subsec:ConvergenceLshaped}}
\label{fig:LshapedDomainFigsJ}
\end{figure}
\subsection{Two-dimensional lid-driven cavity flow}
\label{subsec:2DCavity}
We next test our numerical scheme using the standard 2D lid-driven cavity flow. Within the domain $\Omega = [0, 1]^2$, the top surface is driven at a constant unit speed in the positive $x$-direction, and the remaining boundaries are fixed with no-slip conditions. The magnetic field is taken as $\bB =(0,0,1)$, and the boundary condition for the electric potential $\phi=0$ is used. We run simulations at different Reynolds numbers, testing meshes with varying resolutions. The results are reported for \(192 \times 192\) uniform mesh. All simulations are initiated from zero initial conditions. 

The streamlines for different Reynolds numbers, with \(\kappa = 16\) for all experiments, are given in Figure \ref{fig:2DCavity}. These results are consistent with those reported in \cite{mervegurbuzphd} and \cite{ENUMATHcananhoca}. The iteration count for convergence is given in Table \ref{tab:2DCavity_NbOfIts}.
\begin{table}[H]
    \centering
    \begin{tabular}{|c|c|c|} \hline 
        $\Reynolds$ & Hartmann number  &number of iterations \\ \hline 
        25 & 20 & 50\\ \hline 
         625 & 100 & 109 \\ \hline 
         1406 & 150 & 205 \\ \hline 
        2500 & 200 & 338\\ \hline 
        4761 & 276 & 599 \\ \hline 
        10000 & 400 & 1750 \\ \hline
    \end{tabular}
    \caption{Iterations counts for convergence for 2D lid-driven cavity problem with constant magnetic field \ref{subsec:2DCavity} ($\rho=100$ and $\gamma=100$).}
    \label{tab:2DCavity_NbOfIts}
\end{table}
 \begin{figure}[H]
\centering
 \begin{subfigure}{0.32\textwidth}
\includegraphics[width=\textwidth]{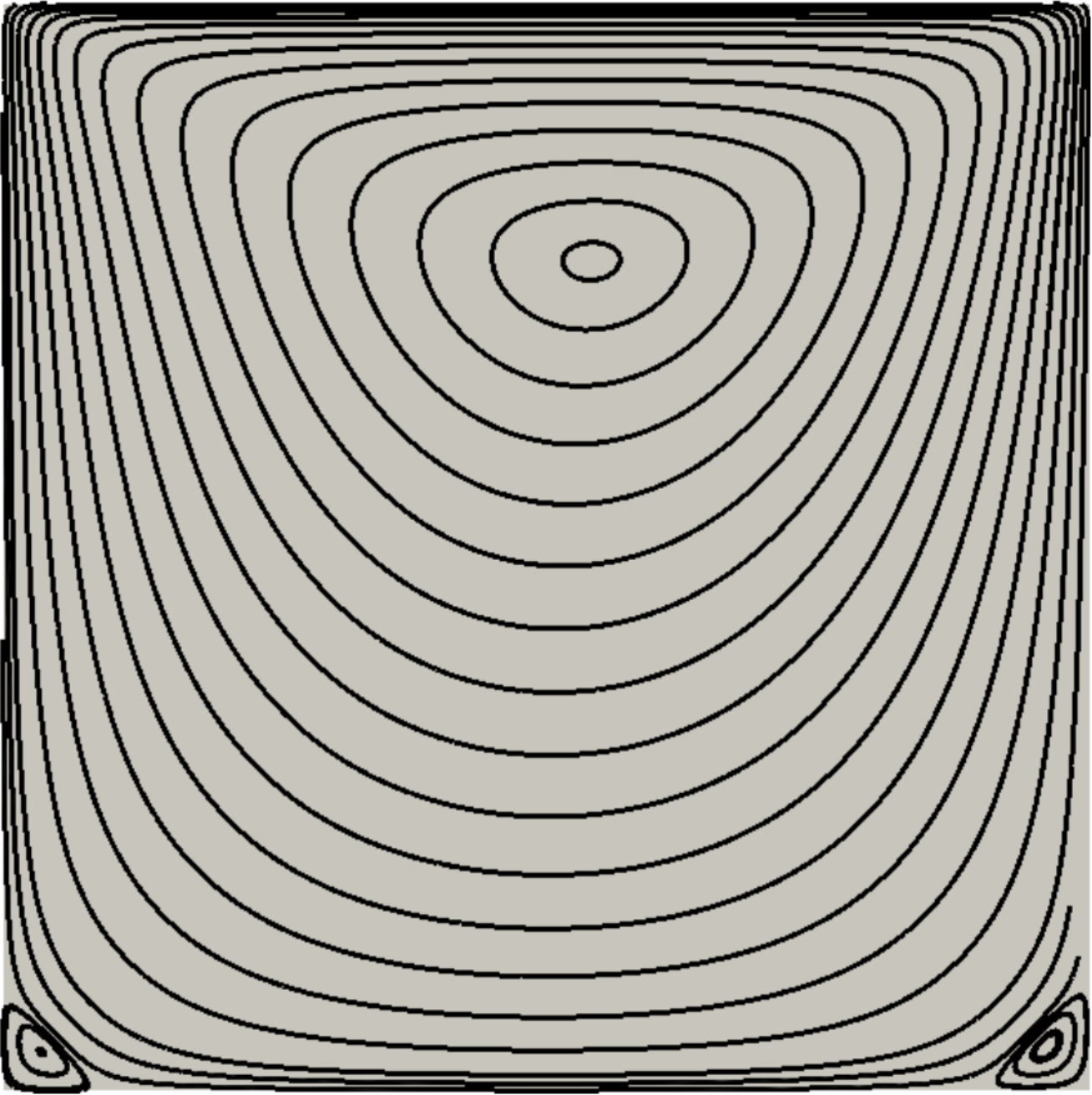}
    \caption{Re=25}
\end{subfigure}
\hfill
\begin{subfigure}{0.32\textwidth}
    \includegraphics[width=\textwidth]{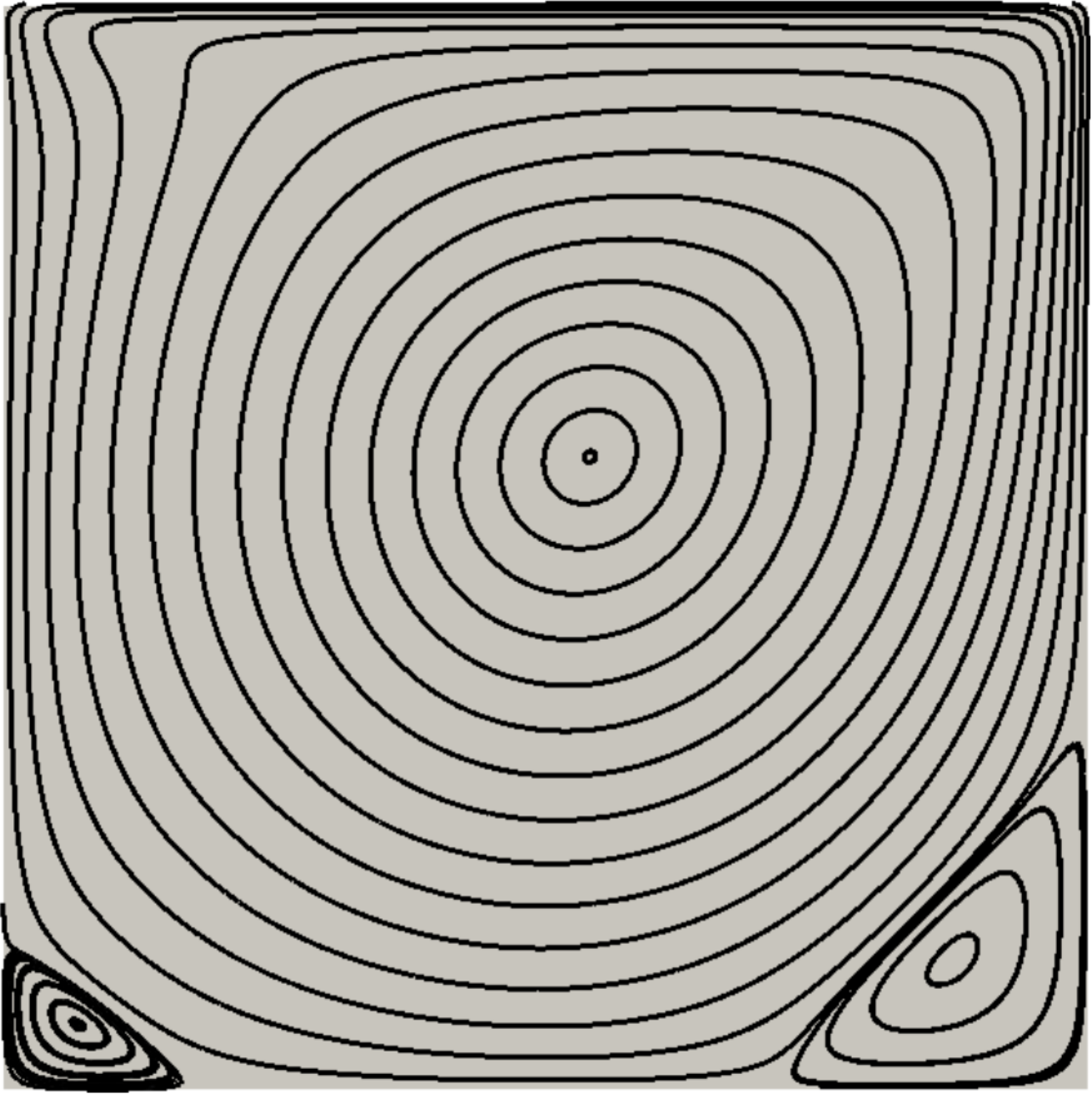}
    \caption{Re=625}
\end{subfigure}
\hfill
\begin{subfigure}{0.32\textwidth}
    \includegraphics[width=\textwidth]{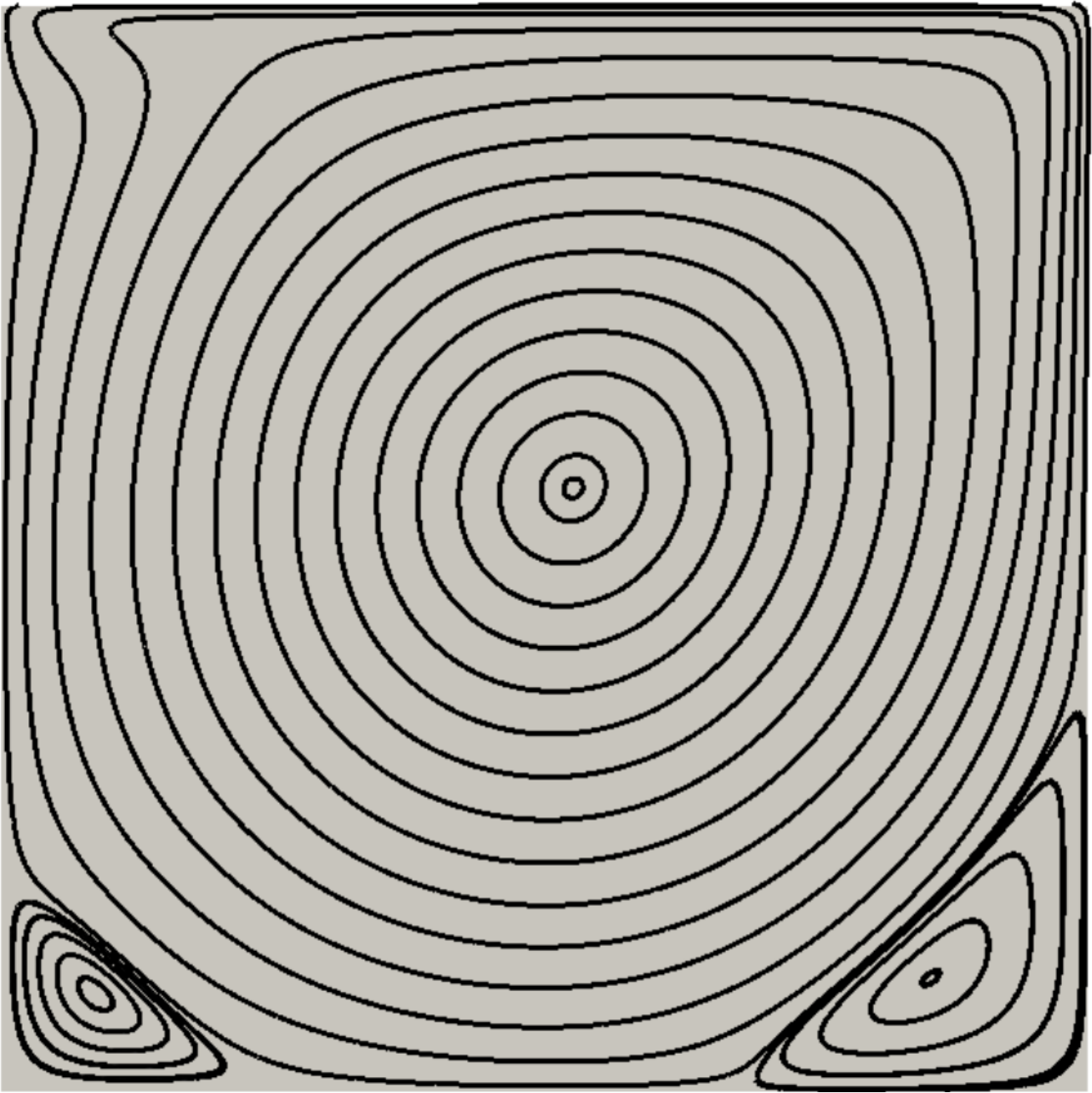}
    \caption{Re=1406}
\end{subfigure}
\\
\hfill
\begin{subfigure}{0.33\textwidth}
    \includegraphics[width=\textwidth]{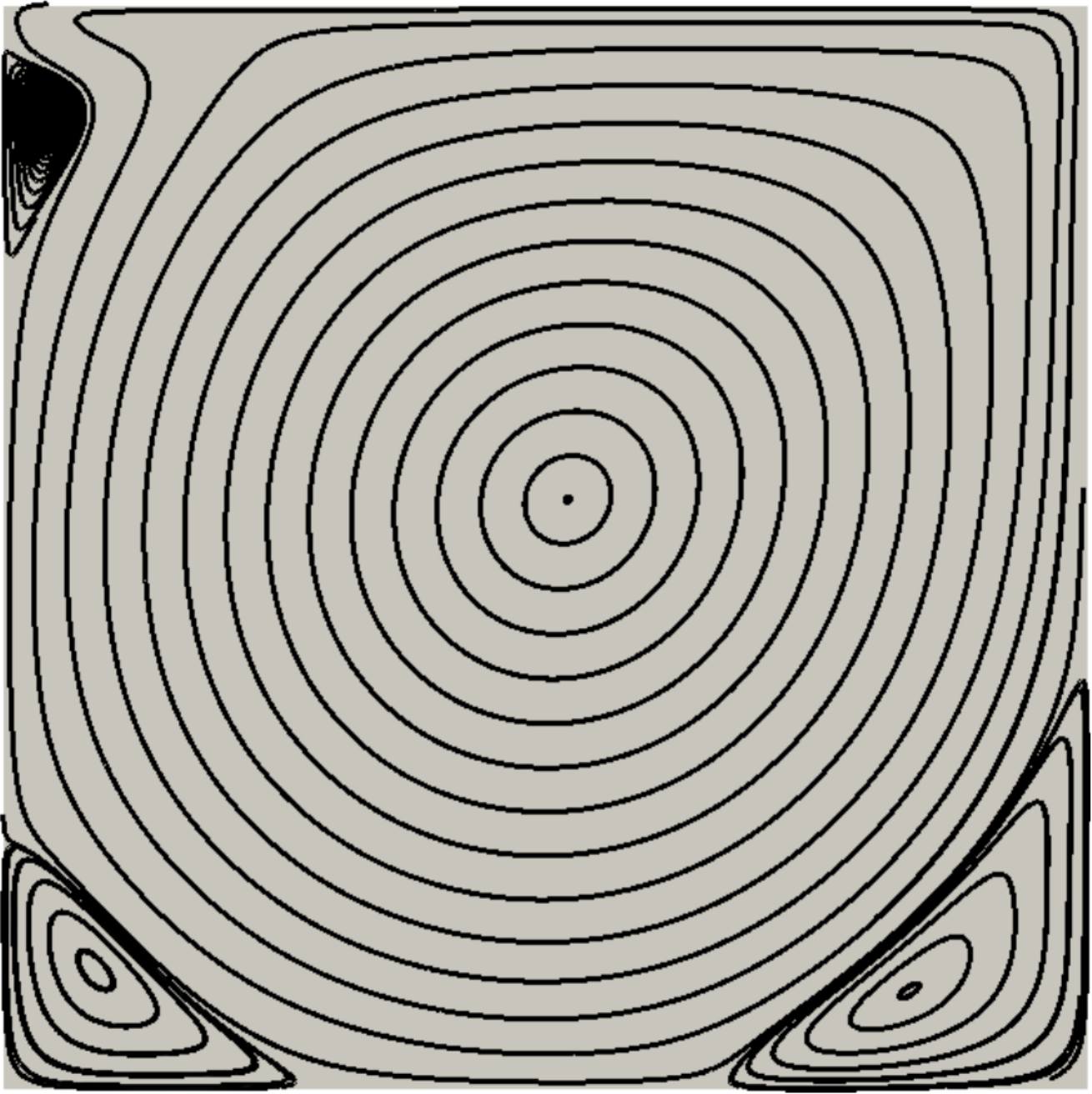}
    \caption{Re=2500}
\end{subfigure}
\hfill
\begin{subfigure}{0.33\textwidth}
    \includegraphics[width=\textwidth]{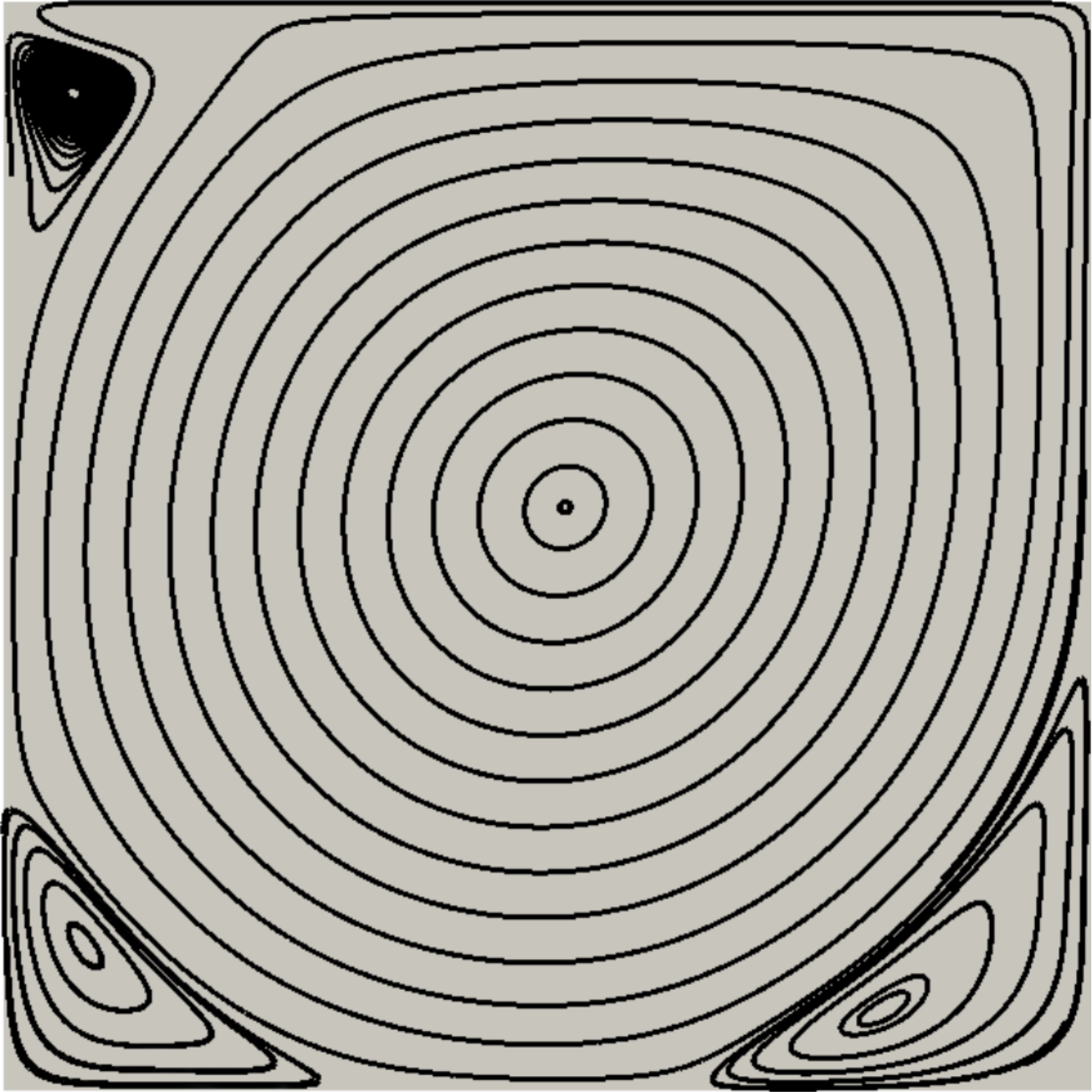}
    \caption{Re=4761}
\end{subfigure}
\hfill
\begin{subfigure}{0.32\textwidth}
    \includegraphics[width=\textwidth]{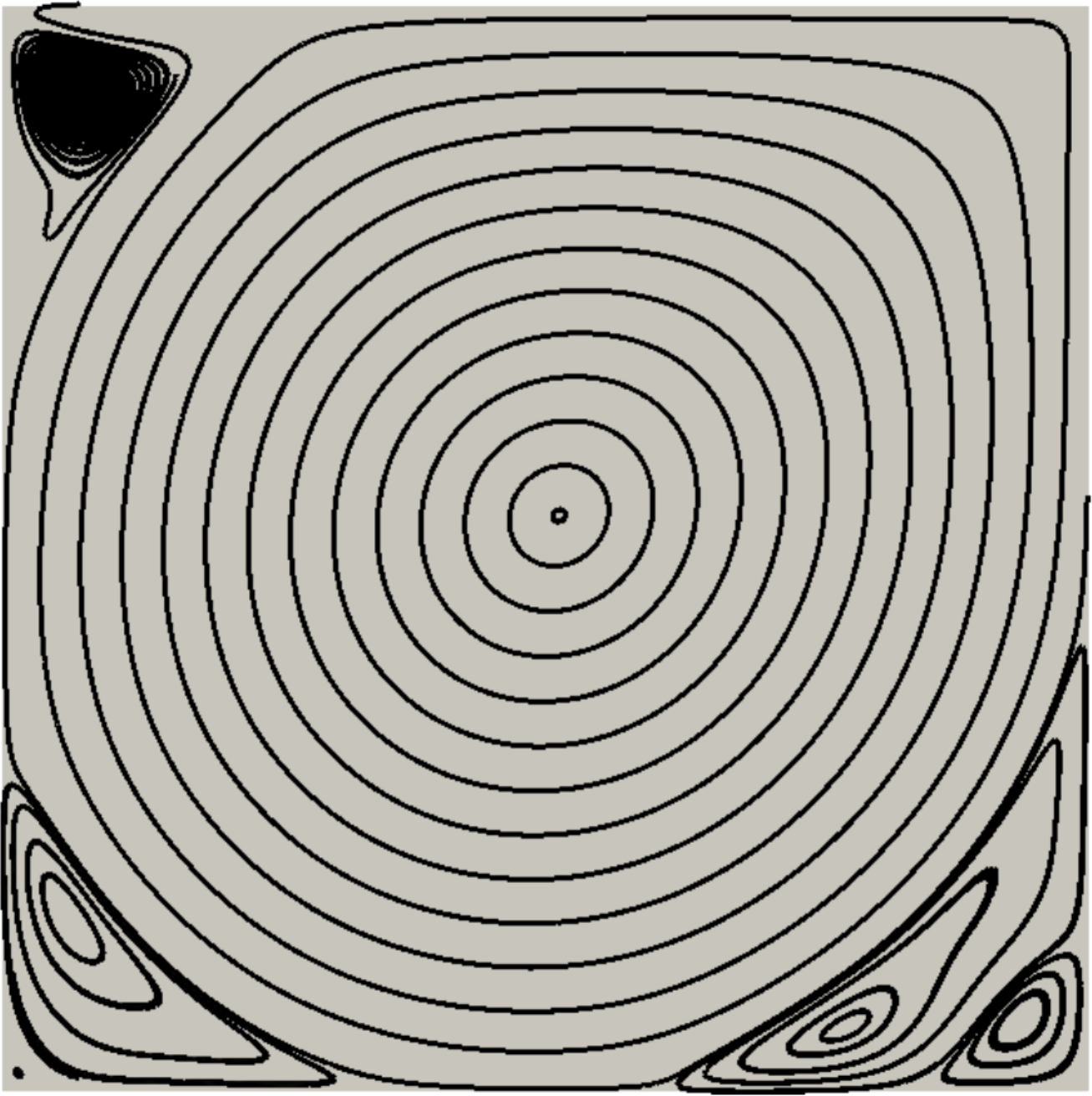}
    \caption{Re=10000}
\end{subfigure}
\caption{Streamlines for different Reynolds numbers for the 2D lid-driven cavity test \ref{subsec:2DCavity} with constant magnetic field ($\rho=100$ and $\gamma=100$).}
\label{fig:2DCavity}
\end{figure}
\subsection{Two-dimensional lid-driven cavity flow with non-constant $\bB$}
\label{subsec:2DCavityNonconstB}
Next, we tested our method on the lid-driven cavity flow problem with a non-constant magnetic field $$\bB = \left(0, 0, \frac{xy}{\sqrt{x^2 + y^2 + 1}}\right), $$
with everything else remaining the same as in the previous test case of Subsection \ref{subsec:2DCavity}. Tests were carried out for $\Reynolds = 100, 200, 400$ and $1000$. The velocity streamlines corresponding to $\Reynolds = 100, 400,$ and $1000$ are given in Figure \ref{fig:2DCavity_nonconstB}, which shows a lot more complex flow structure compared to the constant magnetic field case. 
 \begin{figure}[H]
\centering
 \begin{subfigure}{0.32\textwidth}
\includegraphics[width=\textwidth]{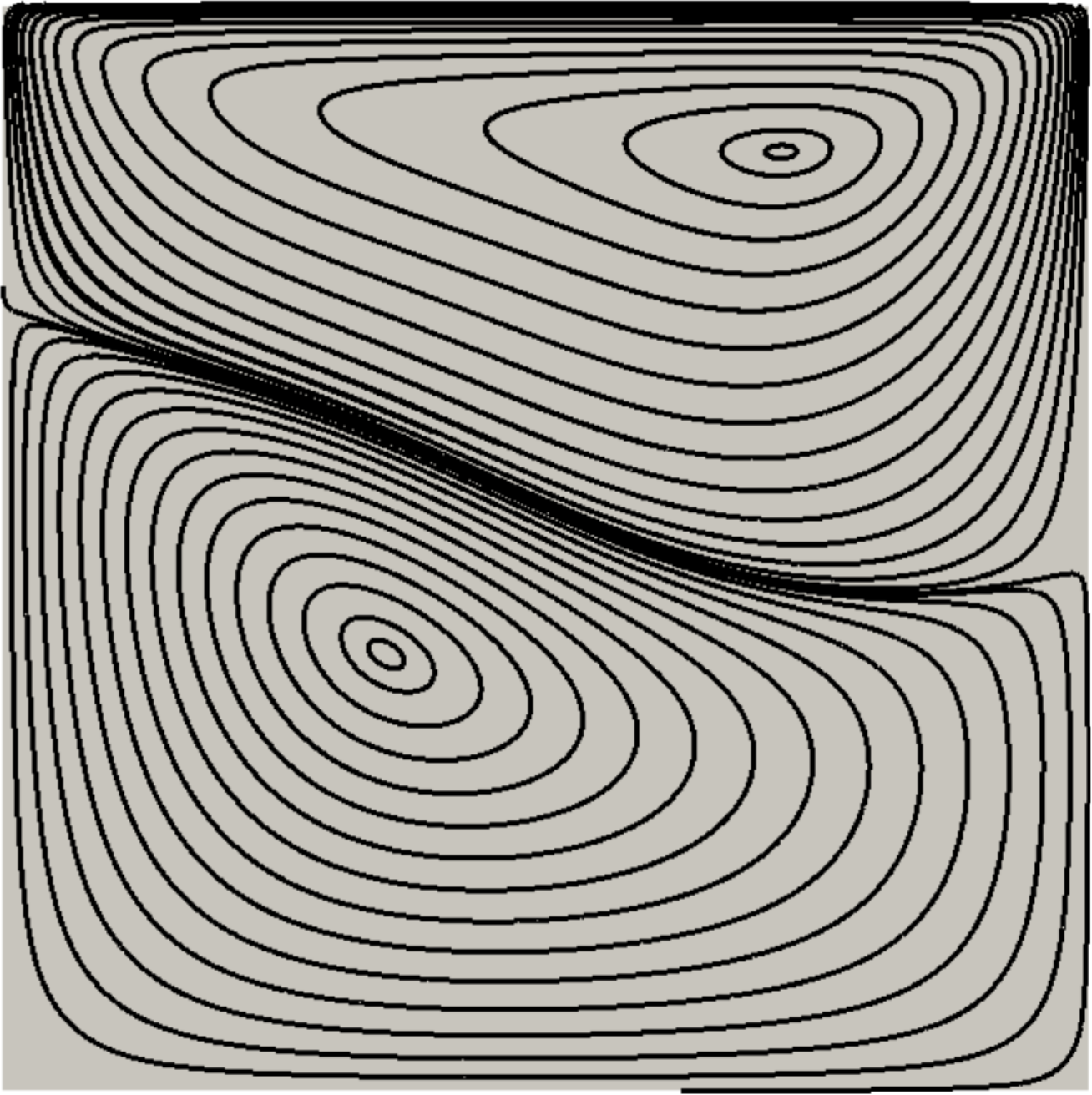}
    \caption{Re=100}
\end{subfigure}
\hfill
 \begin{subfigure}{0.32\textwidth}
\includegraphics[width=\textwidth]{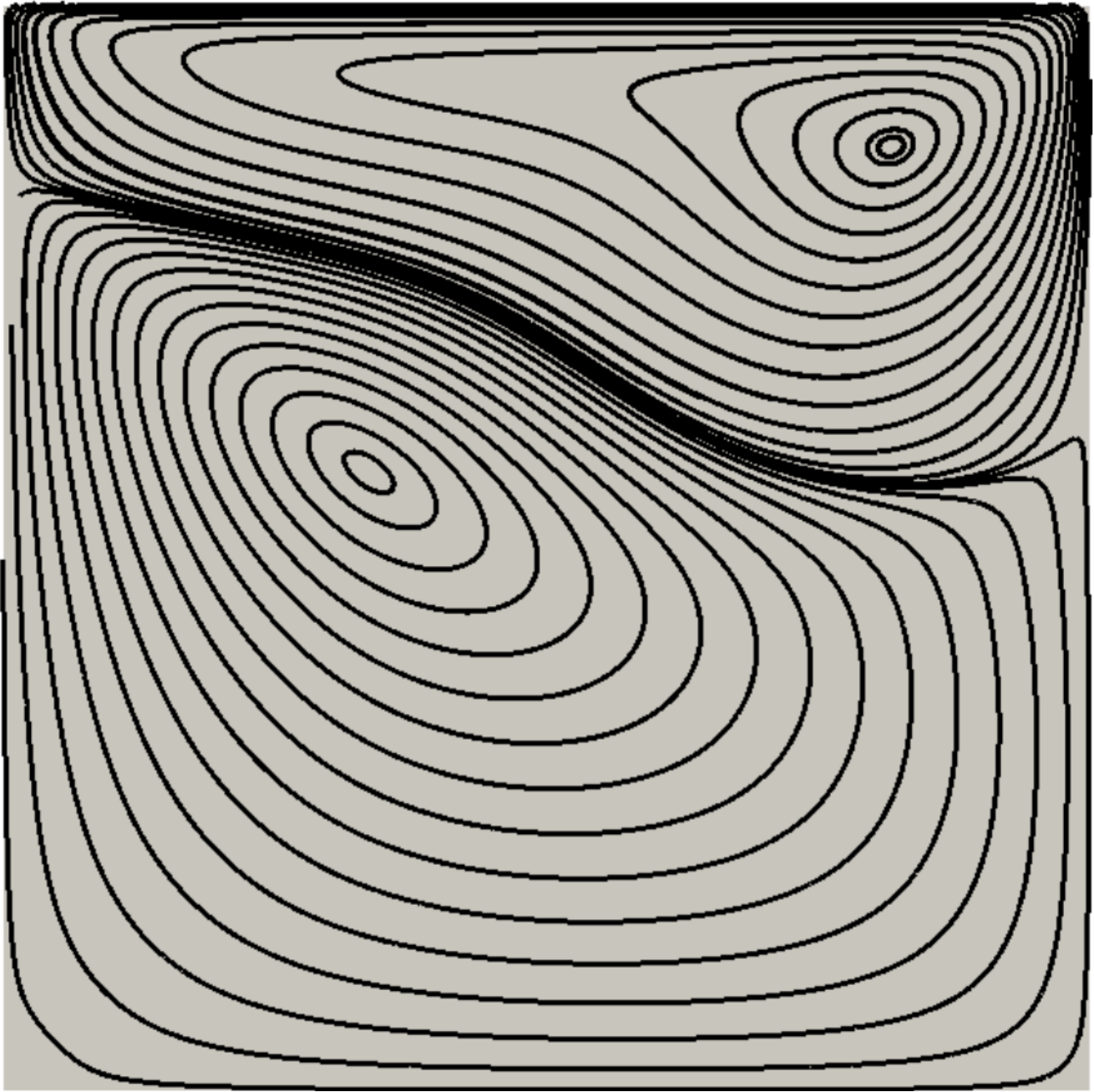}
    \caption{Re=400}
\end{subfigure}
\hfill
\begin{subfigure}{0.32\textwidth}
    \includegraphics[width=\textwidth]{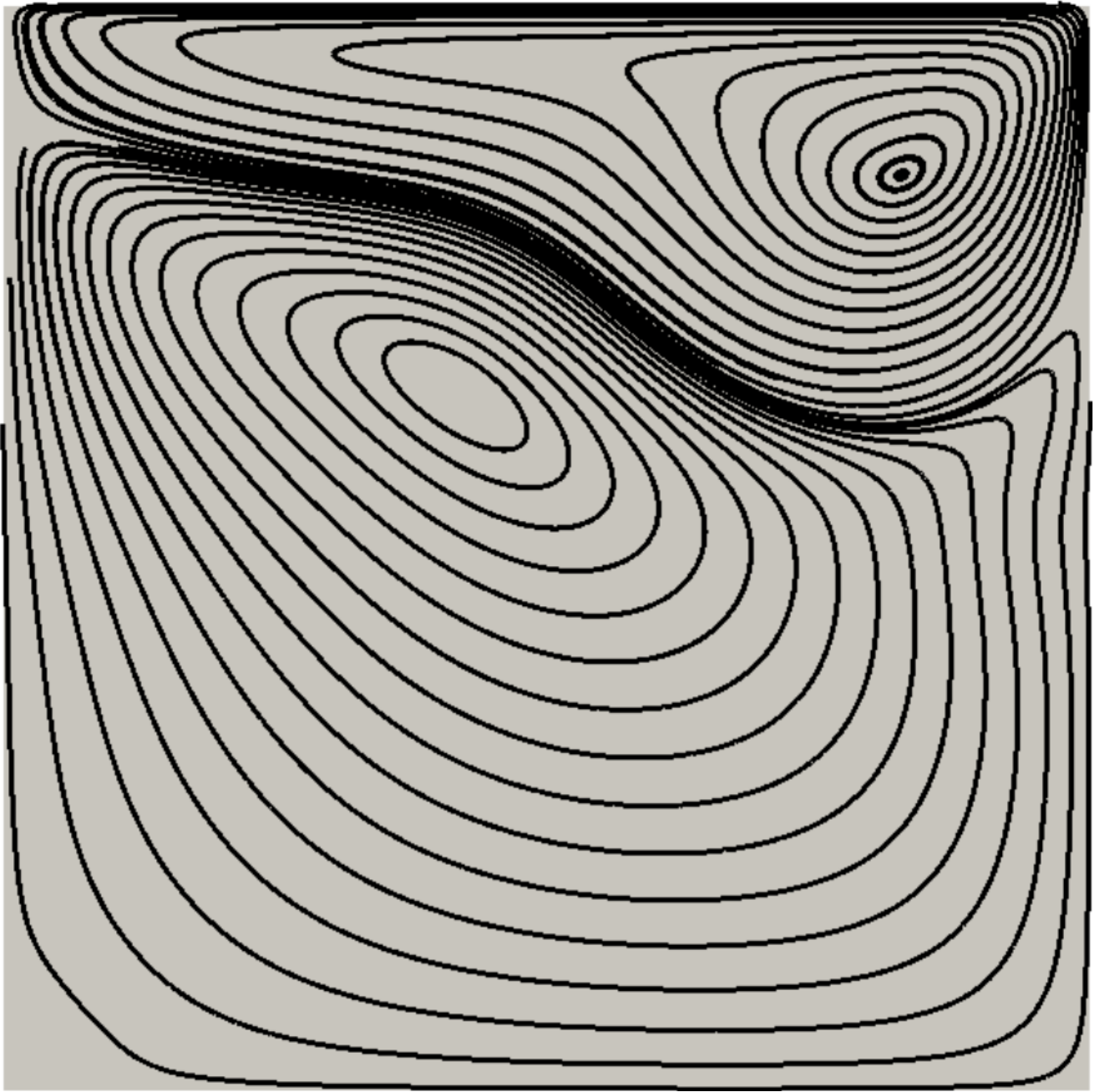}
    \caption{Re=1000}
\end{subfigure}
\caption{Streamlines for different $\Reynolds$ for a non-constant magnetic field}
\label{fig:2DCavity_nonconstB}
\end{figure}
Table \ref{tab:2DNonConstBCavity_NbOfIts} lists the number of iterations it took for convergence. One can notice that in this more complex case, more non-linear iterations are needed to reach the steady state for the same $\Reynolds$ number, as compared to the constant $\bB$ case.
\begin{table}[H]
    \centering
    \begin{tabular}{|c|c|c|} \hline 
        $\Reynolds$ & Hartmann number  &number of iterations \\ \hline 
        100 & 40 & 104\\ \hline 
        200 & 56.6 & 203 \\ \hline 
        400 & 80 & 413 \\ \hline 
        1000 & 126.5 & 1090 \\ \hline
    \end{tabular}
    \caption{Iterations counts for convergence for 2D lid-driven cavity problem with non-constant $\bB$ ($\rho=10$ and $\gamma=1$).}
\label{tab:2DNonConstBCavity_NbOfIts}
\end{table}
\subsection{Three-dimensional lid-driven cavity}
\label{subsec:3DCavity}
Our last experiment is with the 3D lid-driven cavity problem in a unit cube. In this instance, the top surface (which lies on the $z=1$ plane) is driven at a constant unit speed in the positive $x$-direction.  To avoid the reduced regularity induced by the edge singularities, we impose $\bu(x,y,1) = (g(x)g(y),0,0)^T$ using a regularized boundary data proposed in \cite{Frutos2016} with
\begin{align*}
g(\xi) & = 
\begin{cases} 
      1 - \dfrac{1}{4}\left(1-\cos \left(\dfrac{\varepsilon-\xi}{\varepsilon}\pi \right) \right)^2, & \text{ if }0 \le \xi \leq \varepsilon, \\
      1, & \text{ if } \varepsilon < \xi < 1-\varepsilon, \\
      1 - \dfrac{1}{4}\left(1-\cos \left(\dfrac{\xi-(1-\varepsilon)}{\varepsilon}\pi \right) \right)^2, & \text{ if } 1-\varepsilon \le \xi \leq 1,
   \end{cases}
\end{align*}
for $\varepsilon=0.1$. We note that this regularization is different from the commonly used ones for this problem, such as in \cite{Zhang2022}. We report the simulations run on a uniform mesh of size $25 \times 25 \times 25$ with a total of $700,529$ dof. The velocity and current density fields were solved using FGMRES preconditioned by a component-wise algebraic multigrid (GAMG) FieldSplit and an overlapping additive Schwarz method (ASM) with local MUMPS, respectively. The pressure and electric potential fields utilized GMRES with Jacobi-based preconditioners. For $\kappa=1$ runs we used $\rho=100$ and for $\kappa=10$ we used $\rho=10$. These values of $\rho$ gave the fastest convergence rates.

The velocity streamlines on the $y=0.5$ plane are given in Figures \ref{fig:3DCavity_kappa=1} and \ref{fig:3DCavity_kappa=10}, which reveal more and more complex flow structures as $\Reynolds$ or $\kappa$ increase.
 \begin{figure}[H]
\centering
 \begin{subfigure}{0.32\textwidth}
\includegraphics[width=\textwidth]{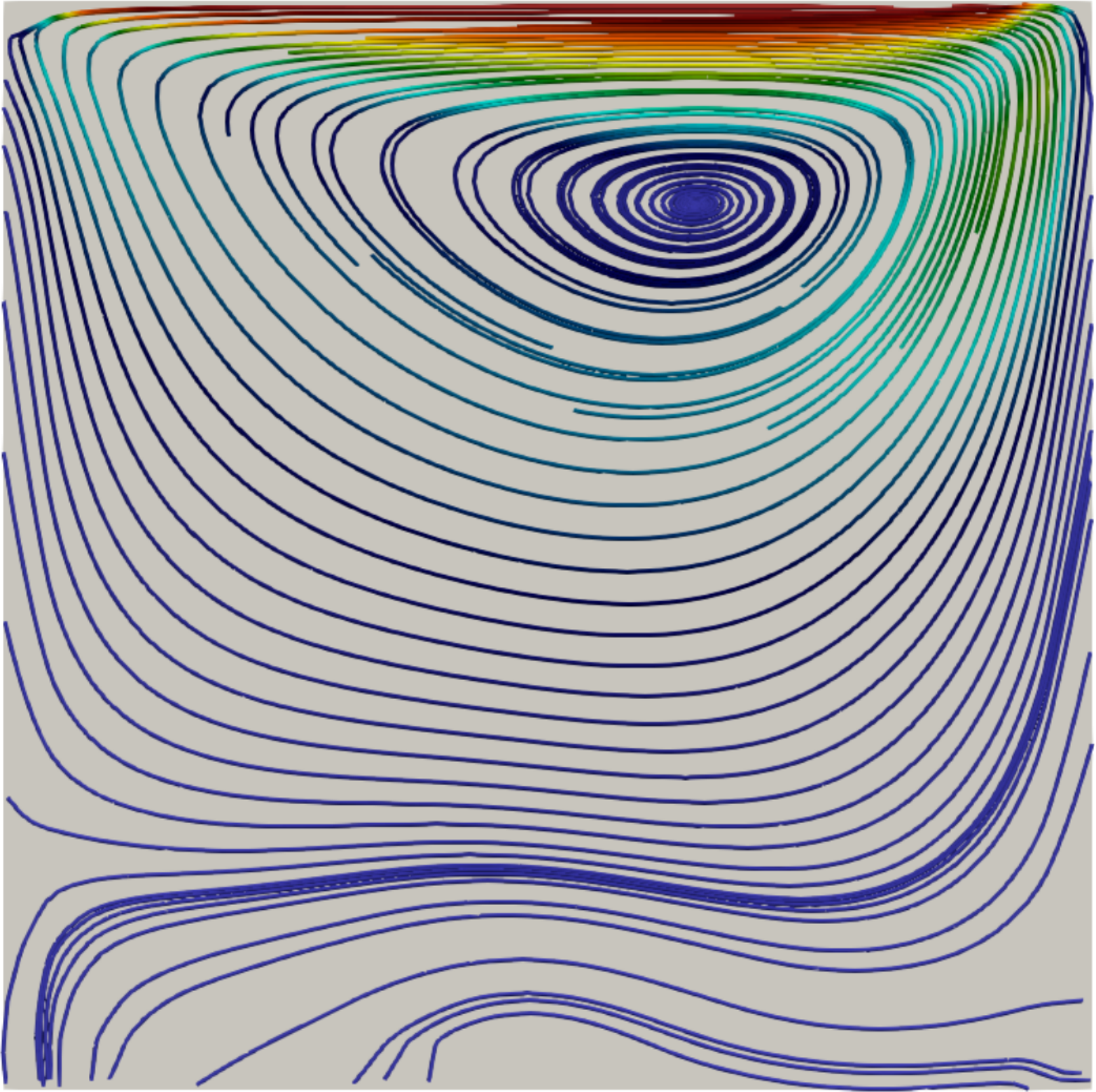}
    \caption{Re=100}
\end{subfigure}
\hfill
\begin{subfigure}{0.32\textwidth}
    \includegraphics[width=\textwidth]{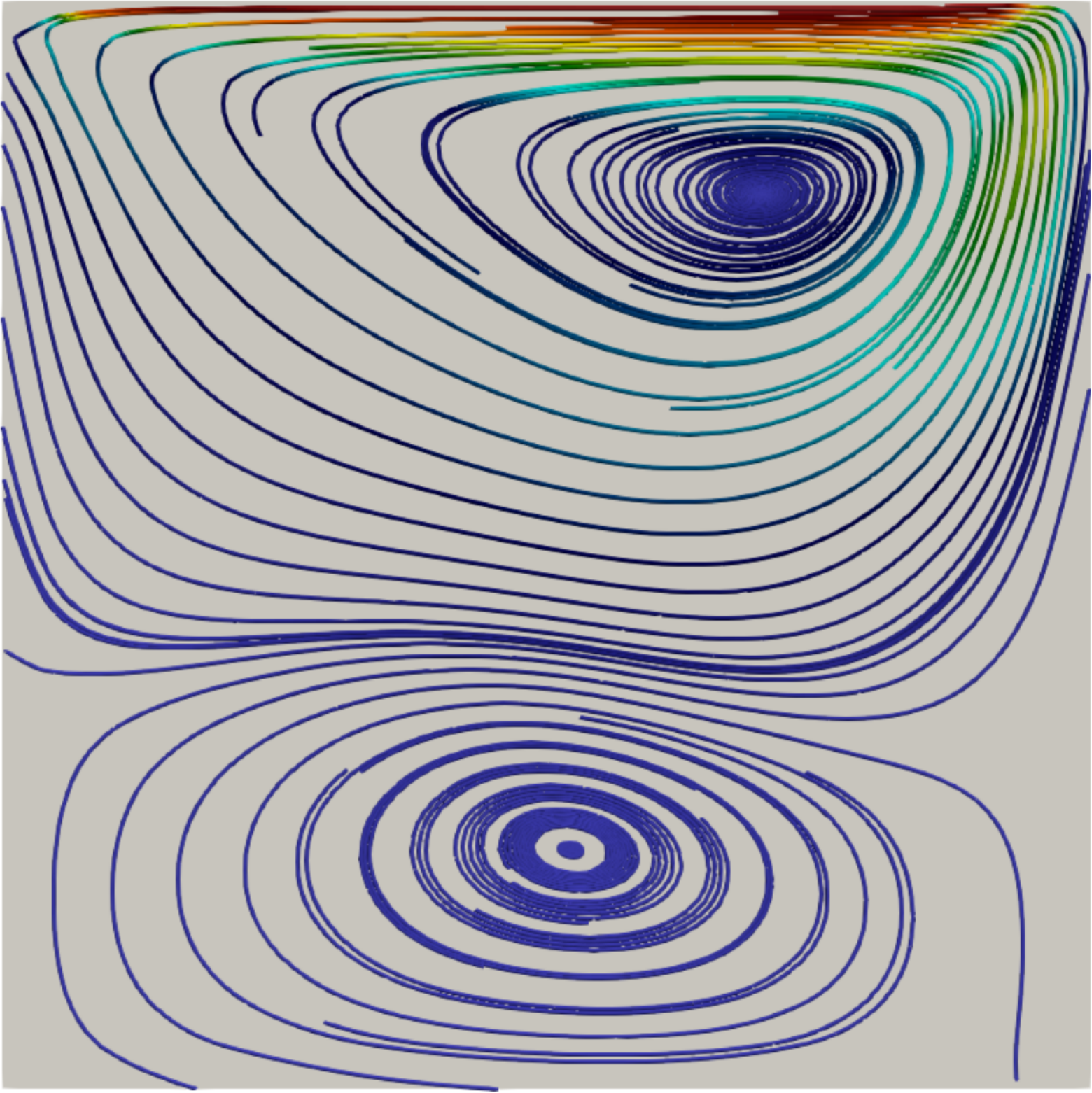}
    \caption{Re=200}
\end{subfigure}
\hfill
\begin{subfigure}{0.32\textwidth}
    \includegraphics[width=\textwidth]{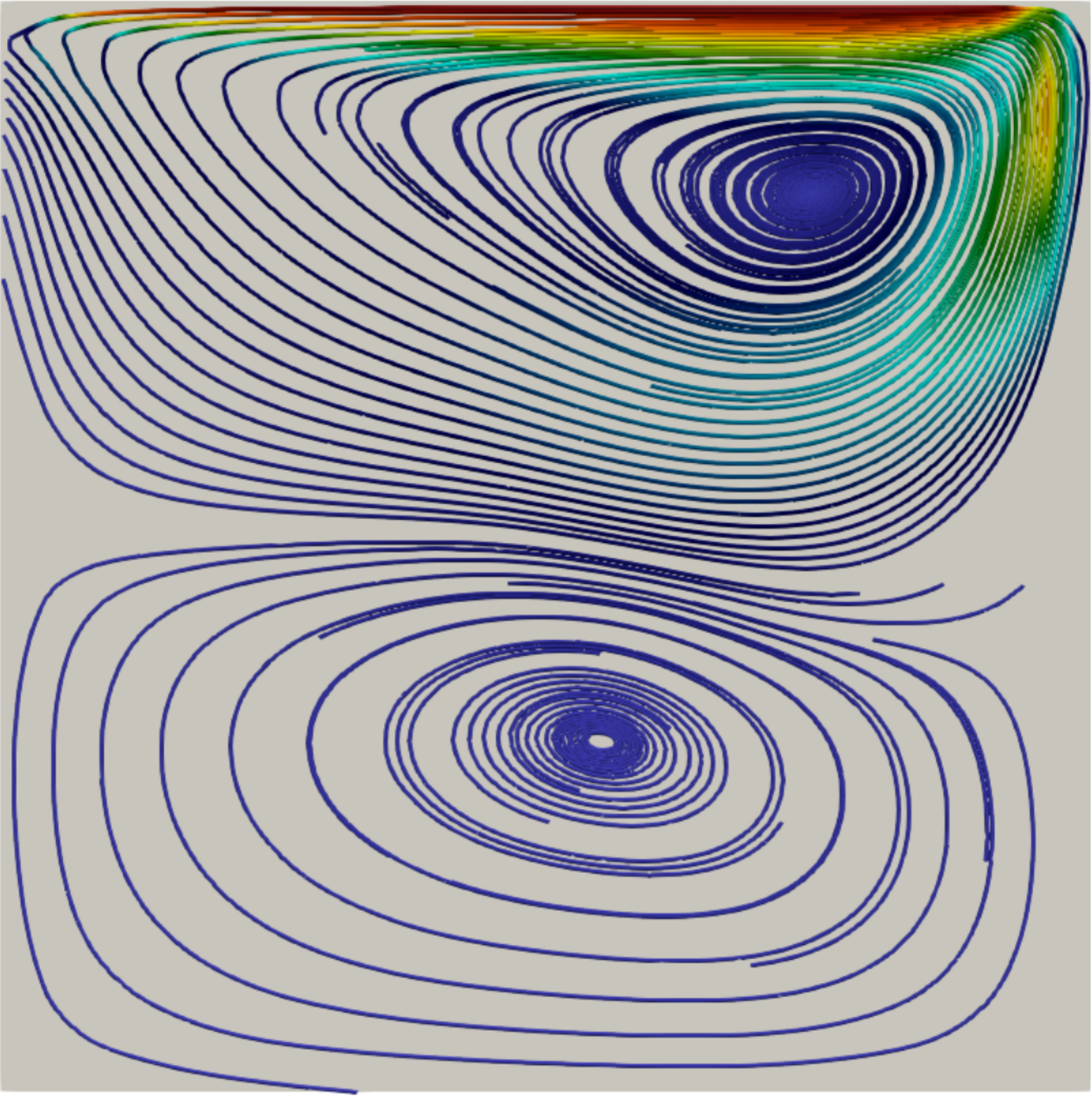}
    \caption{Re=400}
\end{subfigure}
\caption{$y=0.5$ plane streamlines with $\kappa=1$ for three-dimensional lid-driven cavity problem.}
\label{fig:3DCavity_kappa=1}
\end{figure}
\begin{figure}[H]
\centering
\hfill
\begin{subfigure}{0.32\textwidth}
    \includegraphics[width=\textwidth]{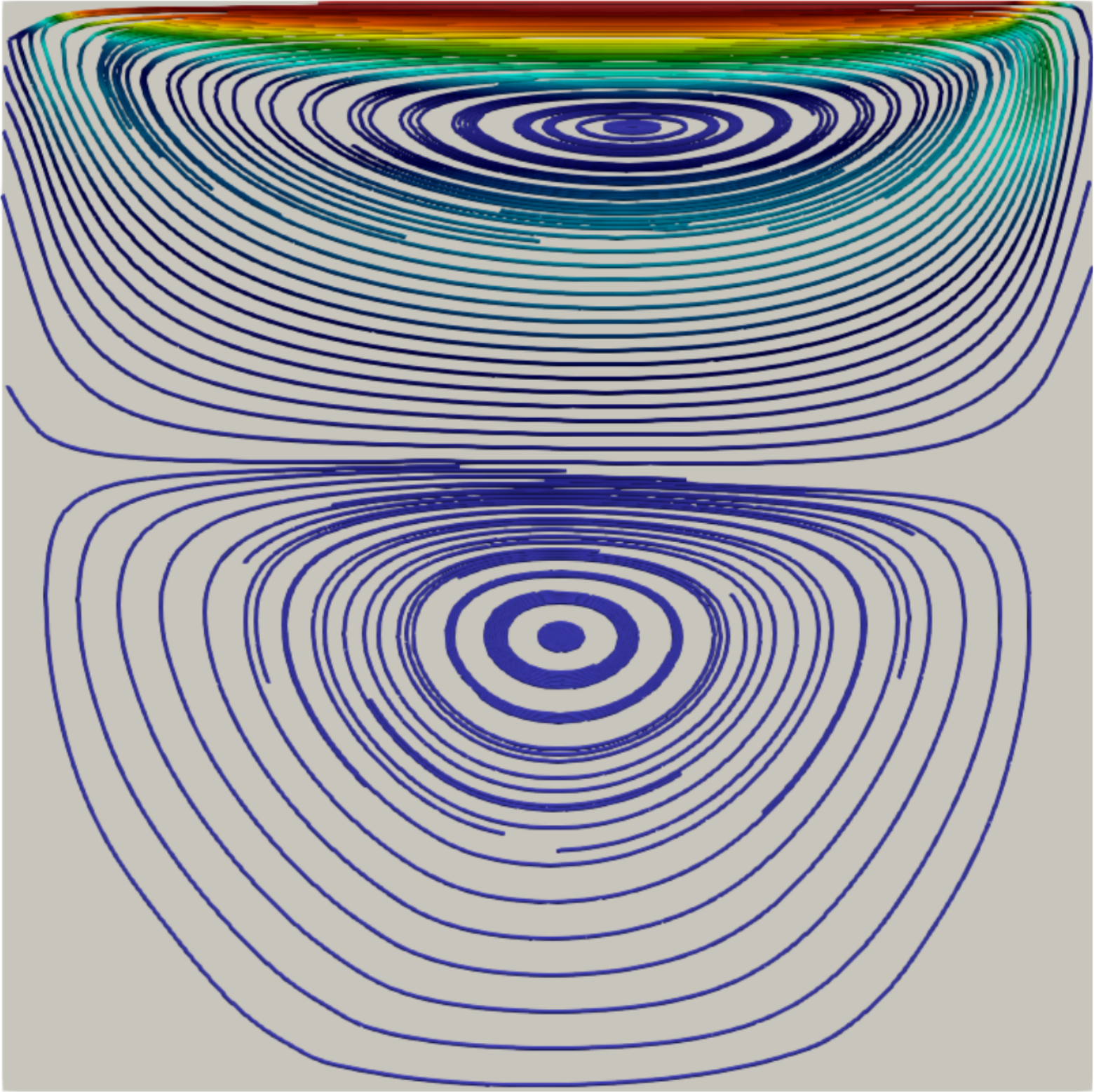}
    \caption{Re=100}
\end{subfigure}
\hfill
\begin{subfigure}{0.32\textwidth}
    \includegraphics[width=\textwidth]{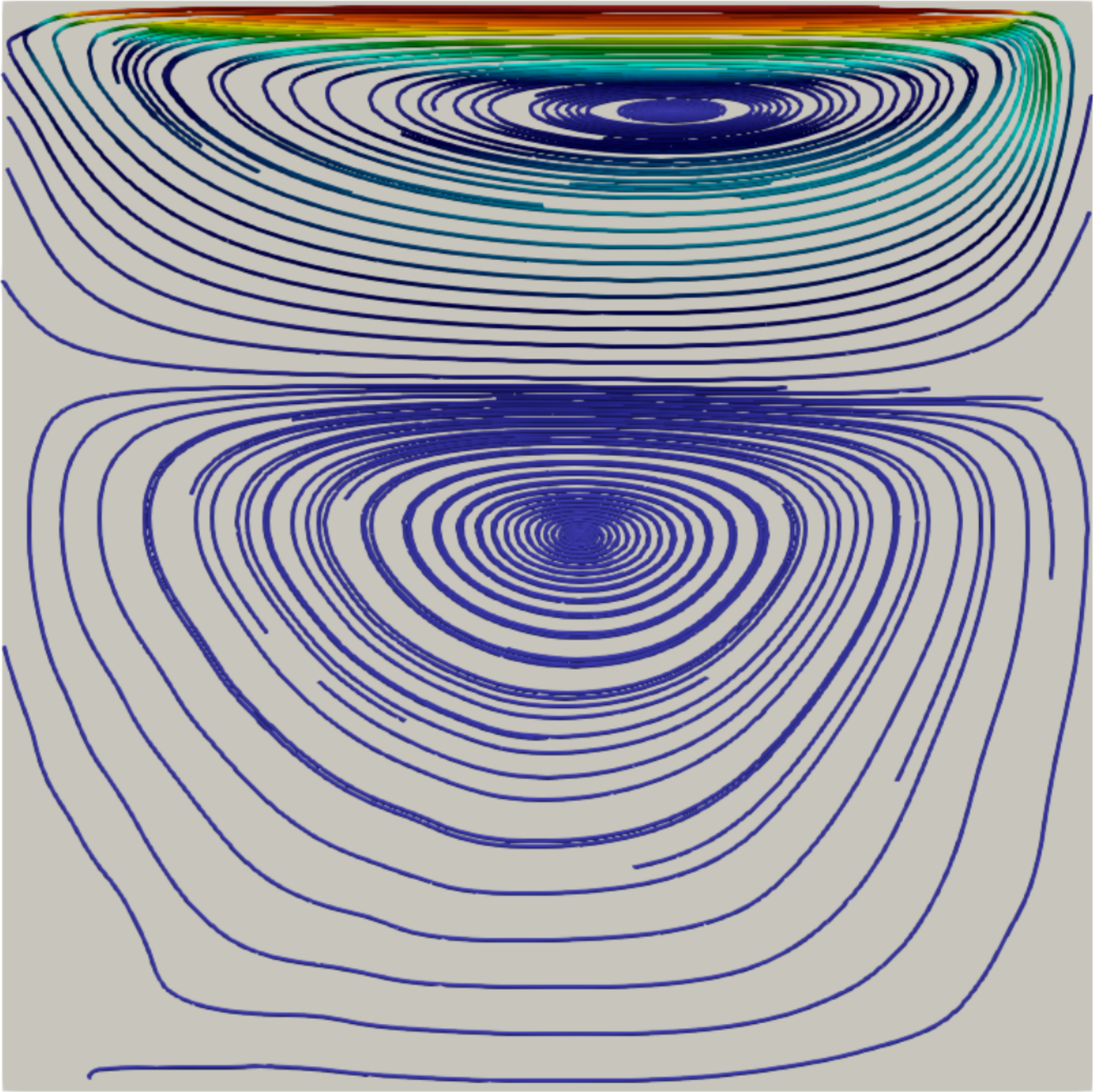}
    \caption{Re=200}
\end{subfigure}
\hfill
\begin{subfigure}{0.32\textwidth}
    \includegraphics[width=\textwidth]{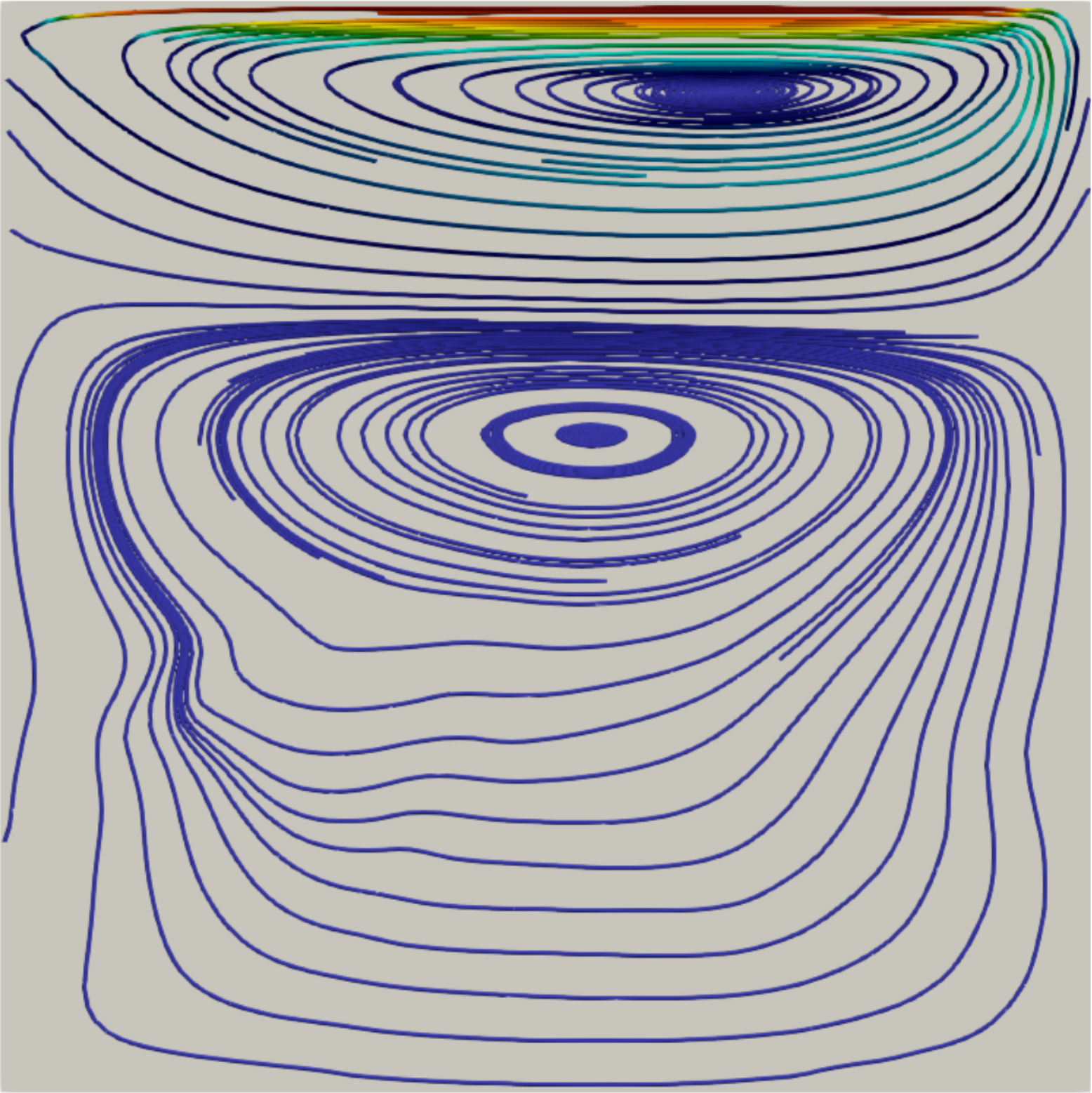}
    \caption{Re=400}
\end{subfigure}
\caption{$y=0.5$ plane streamlines with $\kappa=10$ for three-dimensional lid-driven cavity problem.}
\label{fig:3DCavity_kappa=10}
\end{figure}
\section{Conclusion}
In this paper, we introduce a novel AH method for solving steady Inductionless MHD equations. The momentum and the incompressibility equations are treated using the AH method of \cite{TC23}, while novel streamline-diffusion inspired AH terms are added in the $\bJ$ equation. We have shown theoretical convergence results under expected small data conditions. The resulting iterative scheme decouples the equations for each problem variable. Numerical experiments show that the novel AH method correctly and efficiently simulates benchmark tests.

\bibliography{reference}
\bibliographystyle{plain}
\end{document}